\documentclass{amsart}
\usepackage{amsfonts}
\usepackage{bbm}
\usepackage{mathrsfs}
\usepackage{color}
\usepackage{amsmath}
\usepackage{bigdelim}
\usepackage{multirow}
\usepackage{amssymb}
\usepackage{indentfirst}
\usepackage{CJK}
\usepackage{fancyhdr}
\usepackage{amscd,amssymb,amsmath,graphicx,verbatim,xy, cases}
\usepackage[TS1,OT1,T1]{fontenc}

\newtheorem{theorem}{Theorem}[section]
\newtheorem{lemma}[theorem]{Lemma}
\newtheorem{corollary}[theorem]{Corollary}
\newtheorem{proposition}[theorem]{Proposition}

\theoremstyle{definition}
\newtheorem{definition}[theorem]{Definition}
\newtheorem{example}[theorem]{Example}

\theoremstyle{remark}
\newtheorem{remark}[theorem]{Remark}
\numberwithin{equation}{section}

\begin{document}
	
	\title[Unitary similarity for dual complex matrices and dual quaternion matrices]{Unitary similarity for dual complex matrices and dual quaternion matrices}
	
	\author{Cailing Yao}
	\address{Cailing Yao, School of Medical Informatics, Changchun University of Chinese Medicine, 130117, Changchun, P. R. China}
	\email{yaocl@ccucm.edu.cn}
	\author{Bingzhe Hou}
	\address{Bingzhe Hou, School of Mathematics, Jilin University, 130012, Changchun, P. R. China}
	\email{houbz@jlu.edu.cn}
	\author{Xiaoqi Feng}
	\address{Xiaoqi Feng, School of Mathematics, Jilin University, 130012, Changchun, P. R. China}
	\email{fengxiaoqi2011@qq.com}

	
	
	\begin{abstract}
		We study unitary similarity of dual complex matrices and dual quaternion matrices. First, we
		characterize unitary similarity in terms of the existence of a skew-Hermitian solution to a Sylvester-type matrix equation. For $2\times2$ dual complex matrices, we obtain a complete three-trace criterion when the standard part is	nonnormal and analyze the limitations of trace-word criteria when the standard part is normal. We then prove, for every $n\geq1$, that two $n\times n$ dual quaternion matrices are unitarily similar if and only if their dual complex adjoint	matrices are unitarily similar, and we extend this result to finite families. As a scalar consequence, we derive a finite Trace-word criterion when the standard part is nonreal and exhibit its failure when the standard part is real.
	\end{abstract}
	\maketitle


\section{Introduction}
The skew field $\mathbb{Q}$ of quaternions was discovered by W. Hamilton \cite{H1843} in 1843. Quaternions have extensive applications in many fields, such as multibody dynamics and control \cite{WW}, computer graphics \cite{KS}, aerospace \cite{WGM}, virtual reality \cite{KC}, and so on. Dual quaternions were introduced by Clifford \cite{C} in $1873$. To date, dual quaternions have been widely used in many fields, such as robotics \cite{AC}, 3D motion modelling and control \cite{LL}. The theory of complex and quaternion matrices have already been well developed, including spectral theory, decomposition theory, numerical analysis, generalized inverses and equations. However, the aforementioned theories related to dual quaternion matrices have yet to be fully developed.

Unitary transformations between quaternion matrices constitute an important class of transformations, which inherit many favorable properties from their complex counterparts, including the preservation of inner products and norms. Considerable research has been devoted to unitary transformations of quaternion matrices (see, e.g., \cite{RL14}, \cite{WN55}). Nevertheless, a systematic study of unitary transformations between dual quaternion matrices is still lacking. This paper aims to provide some characterizations of unitary similarity for $n\times n$  dual complex matrices and dual quaternion matrices.

The present article is organized as follows. In Section 2, we introduce the
basic concepts and properties used throughout the paper, including
quaternions and quaternion matrices, dual quaternions and dual quaternion
matrices, complex adjoint matrices, Sylvester equations, the trace of complex
matrices (denoted by $\operatorname{tr}$), and the Trace of quaternion
matrices (denoted by $\operatorname{Tr}$). In Section 3, we characterize
unitary similarity of $n\times n$ dual complex and dual quaternion matrices
by the existence of a skew-Hermitian solution to a specialized Sylvester
matrix equation. In Section 4, we provide a
complete characterization of unitary similarity for $2\times2$
dual complex matrices with nonnormal standard parts, and discuss the case where the standard parts are normal from the perspective of matrix trace words. In Section 5, we
prove, for arbitrary $n$, that two dual quaternion matrices are unitarily
similar if and only if their dual complex adjoint matrices are unitarily
similar. We then derive the scalar Trace-word criterion and explain the
exceptional behavior when the standard part is real.

\section{Preliminaries}
In this section, we introduce some basic definitions and related properties.

\textbf{Dual numbers.}\ \ Denote by $\mathbb{R}$, $\mathbb{C}$ and $\mathbb{Q}$ the sets of all real numbers, complex numbers and quaternions, respectively. Assume $q=q_{st}+q_{\mathcal{I}}\epsilon$, where $q_{st},q_{\mathcal{I}}\in\mathbb{R}$, $\epsilon$ is the infinitesimal unit, satisfying $\epsilon^{2}=0$. We call $q$ a dual number, $q_{st}$ the real part or standard part of $q$ and $q_{\mathcal{I}}$ the dual part or infinitesimal part of $q$. Denote by $\mathbb{D}$ the set of dual numbers. Additionally, if $q_{st},q_{\mathcal{I}}\in\mathbb{C}$ or $q_{st},q_{\mathcal{I}}\in\mathbb{Q}$, respectively, then we say $q$ is a dual complex number or dual quaternion. We denote by $\mathbb{DC}$ the set of dual complex numbers and $\mathbb{DQ}$ the set of dual quaternions, respectively. The infinitesimal unit $\epsilon$ is commutative in multiplication with numbers in $\mathbb{R}$, $\mathbb{C}$ and $\mathbb{Q}$. If $q_{st}\neq 0$, we say $q$ is appreciable, or else, we say $q$ is infinitesimal.

For any two dual numbers $p,q\in \mathbb{D}$, where
$$
p=p_{st}+p_{\mathcal{I}}\epsilon, \ q=q_{st}+q_{\mathcal{I}}\epsilon.
$$
 If either of the following two cases occurs,
\begin{enumerate}
\item [(i).] $p_{st}<q_{st}$,

\item [(ii).] $p_{st}=q_{st}$ and $p_{\mathcal{I}}\leq q_{\mathcal{I}}$,
\end{enumerate}
then we say $p\leq q$. Furthermore, we say $p$ is positive, nonnegative, nonpositive or negative, if $p>0$, $p\geq 0$, $p\leq 0$ or $p<0$, respectively. More about the total order for dual numbers can be seen in \cite{QLY}.

\textbf{Quaternions and quaternion matrices.}\ \ We define $a\in\mathbb{Q}$ to be $a=a_0+a_1\boldsymbol{i}+a_2\boldsymbol{j}+a_3\boldsymbol{k},$ where $a_0, a_1, a_2, a_3\in\mathbb{R}$ and $\boldsymbol{i}^2=\boldsymbol{j}^2=\boldsymbol{k}^2=\boldsymbol{ijk}=-1$. Denote the real and imaginary parts of $a$  by $\mathrm{Re}(a)=a_0$ and  $\mathrm{Im}(a)=a_1\boldsymbol{i}+a_2\boldsymbol{j}+a_3\boldsymbol{k},$ respectively, the modulus of $a$ by
$|a|=\sqrt{a_0^2+a_1^2+a_2^2+a_3^2}$, and the conjugation of $a$ by $\overline{a}=a_0-a_1\boldsymbol{i}-a_2\boldsymbol{j}-a_3\boldsymbol{k}$. We say $a\in\mathbb{Q}$ is invertible if there exists $b\in\mathbb{Q}$ such that $ab=ba=1$ and call $b$ the inverse of $a$.

Denote by $\mathbb{C}^{n\times n}$ and $\mathbb{Q}^{n\times n}$ the sets of all $n\times n$ complex matrices and $n\times n$ quaternion matrices, respectively. If $Q=(q_{ij}) \in \mathbb{Q}^{n\times n}$, we use $\overline{Q}=(\overline{q_{ij}}),Q^{T}=(q_{ji}),Q^{*}=\overline{Q^{T}}=(\overline{Q})^{T}$ to denote the conjugation, transpose and conjugate transpose of $Q$ respectively. If $Q=\overline{Q}$, then $Q$ is a real matrix. If $Q=Q^{T}$, we say $Q$ is symmetric and if $Q=Q^{*}(Q=-Q^{*})$, then we say $Q$ is Hermitian (skew-Hermitian).

For two matrices $A,B\in \mathbb{Q}^{n\times n}$, if there exists an invertible matrix $P\in \mathbb{Q}^{n\times n}$ such that $PAP^{-1}=B$, then we say $A$ is similar to $B$ and denote by $A\sim B$. Especially, if $P$ is unitary, we say $A$ is unitarily similar to $B$ and denote by $A\sim_{u} B$.

\textbf{Dual quaternions.}\ \ For a dual quaternion $q\in\mathbb{DQ}$, it has the form $q=q_{st}+q_{\mathcal{I}}\epsilon$, where $q_{st},q_{\mathcal{I}}\in\mathbb{Q}$ are the standard part and the infinitesimal part of $q$, respectively. The conjugation of $q$ is defined to be $q^{\ast}=q_{st}^{\ast}+q_{\mathcal{I}}^{\ast}\epsilon$. Hence, if $q=q^{\ast}$, then $q\in\mathbb{D}$.

Now we let $p,q\in\mathbb{DQ}$, where $p=p_{st}+p_{\mathcal{I}}\epsilon$ and $q=q_{st}+q_{\mathcal{I}}\epsilon$. Define their addition and multiplication as
$$
p+q=(p_{st}+q_{st})+(p_{\mathcal{I}}+q_{\mathcal{I}})\epsilon,~ pq=p_{st}q_{st}+(p_{st}q_{\mathcal{I}}+p_{\mathcal{I}}q_{st})\epsilon.
$$
A direct calculation shows that a dual number is commutative with any dual quaternion. For $p\in\mathbb{DQ}$, if there exists $q\in\mathbb{DQ}$ such that $pq=qp=1$, then we say $p$ is invertible and $q$ is the inverse of $p$. It follows that $p$ is invertible if and only if $p$ is appreciable and
$$
p^{-1}=p_{st}^{-1}-p_{st}^{-1}p_{\mathcal{I}}p_{st}^{-1}\epsilon.
$$
For $p=p_{st}+p_{\mathcal{I}}\epsilon$, the modulus of $p$ is defined as
\begin{enumerate}
\item if $p_{st}=0$, $|p|=|p_{\mathcal{I}}|\epsilon$, where $|p_{\mathcal{I}}|$ denotes the modulus of the quaternion $p_{\mathcal{I}}$. That is, if we set $p_{\mathcal{I}}=y_{0}+y_{1}\boldsymbol{i}+y_{2}\boldsymbol{j}+y_{3}\boldsymbol{k}$, then
    $$
    |p_{\mathcal{I}}|=\sqrt{y_{0}^{2}+y_{1}^{2}+y_{2}^{2}+y_{3}^{2}};
    $$
\item if $p_{st}\neq 0$, $|p|=|p_{st}|+\frac{p_{st}\cdot p_{\mathcal{I}}}{|p_{st}|}\epsilon$, where $p_{st}\cdot p_{\mathcal{I}}$ is the dot product of $p_{st}$ and $p_{\mathcal{I}}$. Specifically, if we let $p_{st}=x_{0}+x_{1}\boldsymbol{i}+x_{2}\boldsymbol{j}+x_{3}\boldsymbol{k}$, $p_{\mathcal{I}}=y_{0}+y_{1}\boldsymbol{i}+y_{2}\boldsymbol{j}+y_{3}\boldsymbol{k}$, then
$$
p_{st}\cdot p_{\mathcal{I}}=x_{0}y_{0}+x_{1}y_{1}+x_{2}y_{2}+x_{3}y_{3}.
$$
It is obvious that $p_{st}\cdot p_{\mathcal{I}}=\frac{p_{st}p_{\mathcal{I}}^{\ast}+p_{\mathcal{I}}p_{st}^{\ast}}{2}$ and by the definition of modulus for $p\in \mathbb{DQ}$, we can see that $|p|$ is a nonnegative dual number. If $|p|=1$, then $p$ is called a unit dual quaternion. A unit dual quaternion $p$ must be invertible and $p^{-1}=p^{\ast}$ always holds. Further, $|p|=1$ if and only if $|p_{st}|=1$ and $p_{st}\cdot p_{\mathcal{I}}=0$.
\end{enumerate}

\textbf{Complex adjoint.}\ \ The complex adjoint matrix of a quaternion matrix is a classical concept, see, e.g., \cite{Z}. For $A=P_{1}+P_{2}\boldsymbol{i}+P_{3}\boldsymbol{j}+P_{4}\boldsymbol{k}\in\mathbb{Q}^{n\times n}$, where $P_{1}, P_{2}, P_{3}, P_{4}\in\mathbb{R}^{n\times n}$. Let $A_{1}=P_{1}+P_{2}\boldsymbol{i}$ and $A_{2}=P_{3}+P_{4}\boldsymbol{i}$, then $A_{1}, A_{2}\in\mathbb{C}^{n\times n}$ and $A=A_{1}+A_{2}\boldsymbol{j}$. We shall call the $2n\times 2n$ complex matrix
$$
A_{\mathbb{C}}=\begin{pmatrix}
A_{1}&A_{2}\\
-\overline{A_{2}}&\overline{A_{1}}
\end{pmatrix}
$$
the complex adjoint matrix of $A$. It can be seen that $A_{\mathbb{C}}$ is uniquely determined by $A$. The complex adjoint matrix possesses excellent properties \cite{Lee}.
\begin{enumerate}
\item [(1)]$(I_{n})_{\mathbb{C}}=I_{2n}$.
\item [(2)]$(AB)_{\mathbb{C}}=A_{\mathbb{C}}B_{\mathbb{C}}$.
\item [(3)]$(A+B)_{\mathbb{C}}=A_{\mathbb{C}}+B_{\mathbb{C}}$.
\item [(4)]$(A^{\ast})_{\mathbb{C}}=(A_{\mathbb{C}})^{\ast}$.
\item [(5)]$(A^{-1})_{\mathbb{C}}=(A_{\mathbb{C}})^{-1}$, if $A$ is invertible.
\item [(6)]$A_{\mathbb{C}}$ is unitary, Hermitian or normal if and only if $A$ is unitary, Hermitian or normal, respectively.
\end{enumerate}
Motivated by this idea, Y. Chen and L. Zhang introduced the concept of dual complex adjoint matrices for dual quaternion matrices in \cite{CZ}. Suppose that $A=A_{1}+A_{2}\boldsymbol{j}+(A_{3}+A_{4}\boldsymbol{j})\epsilon\in\mathbb{DQ}^{n\times n}$, where $A_{1}, A_{2}, A_{3}, A_{4}\in\mathbb{C}^{n\times n}$. Let
$$
A_{\mathbb{C}}=\begin{pmatrix}
A_{1}&A_{2}\\
-\overline{A_{2}}&\overline{A_{1}}
\end{pmatrix}+\begin{pmatrix}
A_{3}&A_{4}\\
-\overline{A_{4}}&\overline{A_{3}}
\end{pmatrix}\epsilon.
$$
The $2n\times 2n$ dual complex matrix $A_{\mathbb{C}}$ is called the dual complex adjoint matrix of $A$. It can be verified that the dual complex adjoint matrix defined in this way shares properties $(1)-(6)$ as well.

\textbf{Sylvester equation.}\ \ The Sylvester equation stems from Sylvester's early systematic study of linear matrix equations. His work, collected in volumes such as ``The Collected Mathematical Papers of James Joseph Sylvester'', has laid an important foundation for modern linear algebra and control theory. F. De Ter\'{a}n and F. Dopico \cite{Te} gave the  consistency and efficient solution of the $\ast$-Sylvester equation $AX+X^{\ast}B=C$ for $A\in\mathbb{F}^{m\times n},B\in\mathbb{F}^{n\times m},C\in\mathbb{F}^{m\times m}$, where $\mathbb{F}$ is a field with characteristic not $2$. Later, F. De Ter\'{a}n \cite{De}, A. Dmytryshyn and B. K{\aa}gstr\"{o}m \cite{DK} showed the consistency of a system of $\ast$-Sylvester equations $A_{l}X+X^{\ast}B_{l}=C_{l}$ for $l=1,\cdots,k$ in the same field, respectively. Nevertheless, an analogous characterization of the quaternion $\ast$-Sylvester equation has not yet been established.

\textbf{Trace.}\ \ For $A=[a_{ij}]\in\mathbb{C}^{n\times n}$, $B=[b_{ij}]\in\mathbb{Q}^{n\times n}$, the trace of $A$ and the Trace of $B$ are defined to be
$$
\mathrm{tr}(A)=\sum_{i=1}^{n}a_{ii},\ \mathrm{Tr}(B)=\sum_{i=1}^{n}(b_{ii}+\overline{b_{ii}})=2\mathrm{Re}(\sum_{i=1}^{n}b_{ii}).
$$
The trace of complex matrices has some very basic properties, such as linearity, invariance under transpose and conjugate transpose, and invariance under cyclic permutations. The Trace of quaternion matrices also possesses the above properties except for that the linearity with respect to multiplication, the scalar must be taken as a real number. More details may refer to \cite{HR}, \cite{DS08}.

\section{Sylvester equation characterizations of unitary similarity in $\mathbb{DC}^{n\times n}$ and $\mathbb{DQ}^{n\times n}$}
F. De Ter\'{a}n and F. Dopico \cite{Te} provided a necessary and sufficient condition for the existence of solutions of $AX+X^{\ast}B=C$ for rectangular matrices with entries in any field $\mathbb{F}$ with characteristic not $2$. Later, F. De Ter\'{a}n \cite{De}, A. Dmytryshyn and B. K{\aa}gstr\"{o}m \cite{DK} provided a necessary and sufficient condition for the existence of solutions of a system of $\ast$-Sylvester equations $A_{i}X+X^{\ast}B_{i}=C_{i}$ for rectangular matrices with entries in any field $\mathbb{F}$ with characteristic different from two. Here we only present the results for square matrices over the complex field $\mathbb{C}$.
\begin{lemma}[\cite{Te}]\label{2}
Let $A,B,C\in \mathbb{C}^{n\times n}$. Then the matrix equation
$$
AX+X^{\ast}B=C
$$
has a solution if and only if
$\begin{pmatrix}
C&A\\
B&0
\end{pmatrix}$ and $\begin{pmatrix}
0&A\\
B&0
\end{pmatrix}$ are $\ast$-congruent. That is, there exists an invertible matrix $P\in\mathbb{C}^{2n\times 2n}$ such that
$$
P^{\ast}\begin{pmatrix}
C&A\\
B&0
\end{pmatrix}P=\begin{pmatrix}
0&A\\
B&0
\end{pmatrix}.
$$
\end{lemma}

\begin{lemma}[\cite{De}]\label{2014}
Let $A_{i},B_{i},C_{i}\in \mathbb{C}^{n\times n}$ for $i=1,\cdots,l$. Then the following statements are equivalent.
\begin{enumerate}
\item The system of $\ast$-Sylvester equations
\begin{displaymath}
\left \{ \begin{array}{ll}
A_{1}X+X^{\ast}B_{1}=C_{1}\\
\ \ \ \ \vdots\ \ \ \ \ \ \ \vdots\\
A_{l}X+X^{\ast}B_{l}=C_{l}
\end{array} \right.
\end{displaymath}
has a solution $X\in \mathbb{C}^{n\times n}$.
\item There exists an invertible matrix $S\in \mathbb{C}^{2n\times 2n}$ such that
$$
S^{\ast}\begin{pmatrix}
C_{i}&A_{i}\\
B_{i}&0
\end{pmatrix}S=\begin{pmatrix}
0&A_{i}\\
B_{i}&0
\end{pmatrix}, \ \mathrm{for\ all}\  i=1,\cdots,l.
$$
\end{enumerate}
\end{lemma}

In fact, requiring  a skew-Hermitian solution of the equation in Lemma \ref{2} is equivalent to requiring a skew-Hermitian solution of the equation $AX-XB=C$.
We characterize the unitary similarity of two $n\times n$ dual complex matrices in terms of the existence of a skew-Hermitian solution to the equation of the form as in Lemma \ref{2}.

There is a necessary and sufficient characterization of unitary matrices $U\in\mathbb{DQ}^{n\times n}$ in terms of their standard parts and infinitesimal parts.

\begin{lemma}\label{fff}
For a matrix $U=U_{st}+U_{\mathcal{I}}\epsilon\in\mathbb{DQ}^{n\times n}$, $U$ is unitary in $\mathbb{DQ}^{n\times n}$ if and only if
$U_{st}$ is unitary and $U_{st}U_{\mathcal{I}}^{\ast}$ is skew-Hermitian.
\end{lemma}

\begin{theorem}\label{14}
Suppose that $A=A_{st}+A_{\mathcal{I}}\epsilon$ and $B=B_{st}+B_{\mathcal{I}}\epsilon$ are two matrices in $\mathbb{DC}^{n\times n}$, where $A_{st},A_{\mathcal{I}},B_{st},B_{\mathcal{I}}\in\mathbb{C}^{n\times n}$. Then $A$ and $B$ are unitarily similar in $\mathbb{DC}^{n\times n}$ if and only if
\begin{enumerate}
\item $A_{st}$ and $B_{st}$ are unitarily similar in $\mathbb{C}^{n\times n}$. That is, there exists a unitary matrix $U_{st}$ in $\mathbb{C}^{n\times n}$ such that $U_{st}A_{st}U_{st}^{\ast}=B_{st}$.
\item For the same unitary matrix $U_{st}$, the matrix equation
\begin{equation}\label{15}
B_{st}X-XB_{st}=B_{\mathcal{I}}-U_{st}A_{\mathcal{I}}U_{st}^{\ast}
\end{equation}
has a skew-Hermitian solution $X\in\mathbb{C}^{n\times n}$.
\end{enumerate}
\end{theorem}
\begin{proof}
We prove the necessity first. Suppose that $A$ and $B$ are unitarily similar. Then there is a unitary matrix $U=U_{st}+U_{\mathcal{I}}\epsilon\in\mathbb{DC}^{n\times n}$ such that $UAU^{\ast}=B$. More precisely,
\begin{numcases}{}
U_{st}A_{st}U_{st}^{\ast}=B_{st}\label{f}\\
U_{st}A_{st}U_{\mathcal{I}}^{\ast}+U_{\mathcal{I}}A_{st}U_{st}^{\ast}=B_{\mathcal{I}}-U_{st}A_{\mathcal{I}}U_{st}^{\ast}\label{ff}
\end{numcases}{}
By Lemma \ref{fff}, the unitarity of $U$ means that $U_{st}$ is unitary and $U_{st}U_{\mathcal{I}}^{\ast}$ is skew-Hermitian. Hence (\ref{f}) implies fact $1$. Further, it follows from $U_{st}A_{st}U_{st}^{\ast}=B_{st}$ that
$$
U_{st}A_{st}=B_{st}U_{st},\ A_{st}U_{st}^{\ast}=U_{st}^{\ast}B_{st}.
$$
Substituting these identities into (\ref{ff}) yields
$$
B_{st}U_{st}U_{\mathcal{I}}^{\ast}+U_{\mathcal{I}}U_{st}^{\ast}B_{st}=B_{\mathcal{I}}-U_{st}A_{\mathcal{I}}U_{st}^{\ast}.
$$
Let $U_{st}U_{\mathcal{I}}^{\ast}=X$, we obtain fact $2$.

Next we prove the sufficiency. Suppose both fact $1$ and $2$ are satisfied. Let $U_{\mathcal{I}}=X^{\ast}U_{st}$. Then $X=U_{st}U_{\mathcal{I}}^{\ast}$. Let $U=U_{st}+U_{\mathcal{I}}\epsilon\in\mathbb{DC}^{n\times n}$, then it is easy to see that the matrix $U$ is unitary and satisfies $UAU^{\ast}=B$. That is, $A$ and $B$ are unitarily similar.
\end{proof}
\begin{corollary}
Suppose that $A=A_{st}+A_{\mathcal{I}}\epsilon$ and $B=B_{st}+B_{\mathcal{I}}\epsilon$ are two matrices in $\mathbb{DC}^{n\times n}$, where $A_{st},A_{\mathcal{I}},B_{st},B_{\mathcal{I}}\in\mathbb{C}^{n\times n}$. Then $A$ and $B$ are unitarily similar in $\mathbb{DC}^{n\times n}$ if and only if
\begin{enumerate}
\item $A_{st}$ and $B_{st}$ are unitarily similar in $\mathbb{C}^{n\times n}$. That is, there exists a unitary matrix $U_{st}$ in $\mathbb{C}^{n\times n}$ such that $U_{st}A_{st}U_{st}^{\ast}=B_{st}$.
\item For the same unitary matrix $U_{st}$, there exists an invertible matrix $P\in\mathbb{C}^{2n\times 2n}$ such that
$$
P\begin{pmatrix}
B_{\mathcal{I}}-U_{st}A_{\mathcal{I}}U_{st}^{\ast} &B_{st}\\
B_{st} & 0
\end{pmatrix}P^{\ast}=\begin{pmatrix}
0       &B_{st}\\
B_{st} & 0
\end{pmatrix},
$$
$$
P\begin{pmatrix}
0 &I_{n}\\
I_{n} & 0
\end{pmatrix}P^{\ast}=\begin{pmatrix}
0       &I_{n}\\
I_{n} & 0
\end{pmatrix}.
$$
\end{enumerate}
\end{corollary}

Note that if two dual complex matrices $A$ and $B$ in $\mathbb{DC}^{n\times n}$ are unitarily similar, then their standard parts $A_{st}$ and $B_{st}$ must be unitarily similar. There may be multiple unitary matrices $U_{st}$ that establish the unitary similarity between $A_{st}$ and $B_{st}$, but the correct choice of $U_{st}$ is essential for finding a compatible $U_{\mathcal{I}}$ that yields the unitary similarity between $A$ and $B$. An improper choice of $U_{st}$ may lead to the situation where, although $A$ and $B$ are unitarily similar, no $U_{\mathcal{I}}$ matching $U_{st}$ exists.

\begin{example}
Let
$T=
\begin{pmatrix}
1 & 0 \\
0 & 1
\end{pmatrix}
+
\begin{pmatrix}
2 & 0 \\
0 & 1
\end{pmatrix}\epsilon\in\mathbb{DC}^{2\times 2}$. Then $T$ must be unitarily similar to itself. Let
$U_{st}=\frac{\sqrt 2}{2}
\begin{pmatrix}
1 & 1 \\
-1 & 1
\end{pmatrix}\in\mathbb{C}^{2\times 2}$. Then $U_{st}$ is unitary and
$$
U_{st}\begin{pmatrix}
1 & 0 \\
0 & 1
\end{pmatrix}U_{st}^{\ast}=\begin{pmatrix}
1 & 0 \\
0 & 1
\end{pmatrix}.
$$

However, there exists no matrix $U_{\mathcal{I}}\in\mathbb{C}^{2\times 2}$ such that $U=U_{st}+U_{\mathcal{I}}\epsilon\in\mathbb{DC}^{2\times 2}$ is unitary and
$UT=TU$.

\end{example}

\begin{proof}
Suppose that
$U_{\mathcal{I}}=
\begin{pmatrix}
a & b \\
c & d \\
\end{pmatrix}$ satisfies that $UT=TU$, i.e.,
$$\begin{aligned}
& (\begin{pmatrix}
1 & 0 \\
0 & 1
\end{pmatrix}
+
\begin{pmatrix}
2 & 0 \\
0 & 1
\end{pmatrix}\epsilon)
(\frac{\sqrt 2}{2}
\begin{pmatrix}
1 & 1 \\
-1 & 1
\end{pmatrix}
+
\begin{pmatrix}
a & b \\
c & d \\
\end{pmatrix}\epsilon)\\
=&
(\frac{\sqrt 2}{2}
\begin{pmatrix}
1 & 1 \\
-1 & 1
\end{pmatrix}
+
\begin{pmatrix}
a & b \\
c & d \\
\end{pmatrix}\epsilon)
(\begin{pmatrix}
1 & 0 \\
0 & 1
\end{pmatrix}
+
\begin{pmatrix}
2 & 0 \\
0 & 1
\end{pmatrix}\epsilon)
\end{aligned}.
$$
Then we have
$$
\begin{pmatrix}
\sqrt2 +a & \sqrt2+b \\
-\frac{\sqrt 2}{2}+c &\frac{\sqrt 2}{2}+d
\end{pmatrix}
=
\begin{pmatrix}
\sqrt2 +a &\frac{\sqrt 2}{2}+b \\
-\sqrt2+c & \frac{\sqrt 2}{2}+d
\end{pmatrix},$$
but this has no solution.
\end{proof}

\begin{theorem}
Suppose that $A=A_{st}+A_{\mathcal{I}}\epsilon$ and $B=B_{st}+B_{\mathcal{I}}\epsilon$ are two matrices in $\mathbb{DQ}^{n\times n}$, where $A_{st},A_{\mathcal{I}},B_{st},B_{\mathcal{I}}\in\mathbb{Q}^{n\times n}$. Then $A$ and $B$ are unitarily similar in $\mathbb{DQ}^{n\times n}$ if and only if
\begin{enumerate}
\item $A_{st}$ and $B_{st}$ are unitarily similar in $\mathbb{Q}^{n\times n}$. That is, there exists a unitary matrix $U_{st}$ in $\mathbb{Q}^{n\times n}$ such that $U_{st}A_{st}U_{st}^{\ast}=B_{st}$.
\item For the same unitary matrix $U_{st}$, the matrix equation
$$
B_{st}X-XB_{st}=B_{\mathcal{I}}-U_{st}A_{\mathcal{I}}U_{st}^{\ast}
$$
has a skew-Hermitian solution $X\in\mathbb{Q}^{n\times n}$.
\end{enumerate}
\end{theorem}
\begin{proof}
The proof is similar to that of Theorem \ref{14}.
\end{proof}

L. Huang \cite{HL1996} proved that for matrices $A,B,C$  in $\mathbb{C}^{2n\times 2n}$ with complex adjoint forms, the matrix equation $AX-XB=C$ has a solution if and only if it has a solution with complex adjoint form. We obtain an analogous result for the matrix equation $AX+X^{\ast}B=C$ where
$A,B,C$ in $\mathbb{C}^{2n\times 2n}$ are of complex adjoint forms.
\begin{lemma}\label{ha}
Suppose that $A,B,C$ are matrices in $\mathbb{C}^{2n\times 2n}$ with complex adjoint forms. Then the following statements are equivalent.
\begin{enumerate}
\item [(a)] The matrix equation $AX+X^{\ast}B=C$ has a solution $X$ in $\mathbb{C}^{2n\times 2n}$.
\item [(b)] The matrix equation $AX+X^{\ast}B=C$ has a solution $X$ with complex adjoint form in $\mathbb{C}^{2n\times 2n}$.
\end{enumerate}
\end{lemma}
\begin{proof}
$(b)\Rightarrow(a)$ is obvious. We only show the proof of $(a)\Rightarrow(b)$. Assume
$$
A=\begin{pmatrix}
A_{1}&A_{2}\\
-\overline{A_{2}}&\overline{A_{1}}
\end{pmatrix}, B=\begin{pmatrix}
B_{1}&B_{2}\\
-\overline{B_{2}}&\overline{B_{1}}
\end{pmatrix}, C=\begin{pmatrix}
C_{1}&C_{2}\\
-\overline{C_{2}}&\overline{C_{1}}
\end{pmatrix}.$$
Since the solution $X$ is in $\mathbb{C}^{2n\times 2n}$, we may partition $X$ into blocks, where each block is an $n\times n$ complex matrix. That is,
$$
X=\begin{pmatrix}
X_{1}&X_{2}\\
X_{3}&X_{4}
\end{pmatrix}.
$$
Then equation $AX+X^{\ast}B=C$ is equivalent to

\begin{displaymath}
\left \{ \begin{array}{ll}
(A_{1}X_{1}+A_{2}X_{3})+(X_{1}^{\ast}B_{1}-X_{3}^{\ast}\overline{B_{2}})=C_{1}\\
(A_{1}X_{2}+A_{2}X_{4})+(X_{1}^{\ast}B_{2}+X_{3}^{\ast}\overline{B_{1}})=C_{2}\\
(-\overline{A_{2}}X_{1}+\overline{A_{1}}X_{3})+(X_{2}^{\ast}B_{1}-X_{4}^{\ast}\overline{B_{2}})=-\overline{C_{2}}\\
(-\overline{A_{2}}X_{2}+\overline{A_{1}}X_{4})+(X_{2}^{\ast}B_{2}+X_{4}^{\ast}\overline{B_{1}})=\overline{C_{1}}
\end{array} \right.
\end{displaymath}
If we assume
$$
X^{\prime}=\begin{pmatrix}
\overline{X_{4}}&-\overline{X_{3}}\\
-\overline{X_{2}}&\overline{X_{1}}
\end{pmatrix},
$$ then it can be verified that $X^{\prime}$ is still a solution of equation $AX+X^{\ast}B=C$. Therefore, $Y=\frac{X+X^{\prime}}{2}$ is also a solution of $AX+X^{\ast}B=C$. Notice that $Y$ has the complex adjoint form, which indicates our conclusion.
\end{proof}

X. Liu, Q. Wang et al. \cite{somany} characterized the consistency of quaternion $\ast$-Sylvester equation $AX+X^{\ast}B=C$ from the aspect of real representation. Here we use Lemma \ref{ha} to give a new  proof of this solvability criterion.

\begin{theorem}\label{201}
Suppose that $A, B, C\in\mathbb{Q}^{n\times n}$. Then the following statements are equivalent.
\begin{enumerate}
\item [(1)]Quaternion $\ast$-Sylvester equation $AX+X^{\ast}B=C$ has a solution.
\item [(2)]$\begin{pmatrix}
C&A\\
B&0
\end{pmatrix}$ and $\begin{pmatrix}
0&A\\
B&0
\end{pmatrix}$ are $\ast$-congruent. That is, there exists an invertible matrix $P\in\mathbb{Q}^{2n\times 2n}$ such that
$$
P^{\ast}\begin{pmatrix}
C&A\\
B&0
\end{pmatrix}P=\begin{pmatrix}
0&A\\
B&0
\end{pmatrix}.
$$
\end{enumerate}
\end{theorem}
\begin{proof}
First we prove $(1)\Rightarrow(2)$. Suppose that $X\in\mathbb{Q}^{n\times n}$ is a solution of $AX+X^{\ast}B=C$ and let
$$
P=\begin{pmatrix}
I_{n}&0\\
-X&I_{n}
\end{pmatrix},
$$
then we can obtain
$$
P^{\ast}\begin{pmatrix}
C&A\\
B&0
\end{pmatrix}P=\begin{pmatrix}
0&A\\
B&0
\end{pmatrix}.
$$

Now we prove $(2)\Rightarrow(1)$. Assume $P$ is the invertible matrix in $\mathbb{Q}^{2n\times 2n}$ such that
$$
P^{\ast}\begin{pmatrix}
C&A\\
B&0
\end{pmatrix}P=\begin{pmatrix}
0&A\\
B&0
\end{pmatrix}.
$$
By taking the complex adjoint of the above equation, we obtain
$$
P_{\mathbb{C}}^{\ast}\begin{pmatrix}
C&A\\
B&0
\end{pmatrix}_{\mathbb{C}}P_{\mathbb{C}}=\begin{pmatrix}
0&A\\
B&0
\end{pmatrix}_{\mathbb{C}}.
$$
Let
\begin{equation}\label{00}
Q=\begin{pmatrix}
I_{n}&0&0&0\\
0&0&I_{n}&0\\
0&I_{n}&0&0\\
0&0&0&I_{n}
\end{pmatrix}\in\mathbb{C}^{4n\times 4n},
\end{equation}

then $Q$ is unitary and
$$
Q\begin{pmatrix}
C_{\mathbb{C}}&A_{\mathbb{C}}\\
B_{\mathbb{C}}&0
\end{pmatrix}Q^{\ast}=\begin{pmatrix}
C&A\\
B&0
\end{pmatrix}_{\mathbb{C}}.
$$
Furthermore, we obtain
$$
(Q^{\ast}P_{\mathbb{C}}^{\ast}Q)\begin{pmatrix}
C_{\mathbb{C}}&A_{\mathbb{C}}\\
B_{\mathbb{C}}&0
\end{pmatrix}(Q^{\ast}P_{\mathbb{C}}^{\ast}Q)^{\ast}=\begin{pmatrix}
0&A_{\mathbb{C}}\\
B_{\mathbb{C}}&0
\end{pmatrix}.
$$
That is, $\begin{pmatrix}
C_{\mathbb{C}}&A_{\mathbb{C}}\\
B_{\mathbb{C}}&0
\end{pmatrix}$ and $\begin{pmatrix}
0&A_{\mathbb{C}}\\
B_{\mathbb{C}}&0
\end{pmatrix}$ are $\ast$-congruent. By Lemma \ref{2}, the matrix equation $A_{\mathbb{C}}Y+Y^{\ast}B_{\mathbb{C}}=C_{\mathbb{C}}$ has a solution. According to Lemma \ref{ha}, the quaternion $\ast$-Sylvester equation $AX+X^{\ast}B=C$ has a solution.
\end{proof}

\begin{corollary}
Let $A_{i},B_{i},C_{i}\in \mathbb{Q}^{n\times n}$ for $i=1,\cdots,l$. Then the following statements are equivalent.
\begin{enumerate}
\item The system of $\ast$-Sylvester equations
\begin{displaymath}
\left \{ \begin{array}{ll}
A_{1}X+X^{\ast}B_{1}=C_{1}\\
\ \ \ \ \vdots\ \ \ \ \ \ \ \vdots\\
A_{l}X+X^{\ast}B_{l}=C_{l}
\end{array} \right.
\end{displaymath}
has a solution $X\in \mathbb{Q}^{n\times n}$.
\item There exists an invertible matrix $S\in \mathbb{Q}^{2n\times 2n}$ such that
$$
S^{\ast}\begin{pmatrix}
C_{i}&A_{i}\\
B_{i}&0
\end{pmatrix}S=\begin{pmatrix}
0&A_{i}\\
B_{i}&0
\end{pmatrix}, \ \mathrm{for\ all}\  i=1,\cdots,l.
$$
\end{enumerate}
\end{corollary}
\begin{proof}
The conclusion follows from Lemma \ref{2014} and Theorem \ref{201}.
\end{proof}
\begin{corollary}
Suppose that $A=A_{st}+A_{\mathcal{I}}\epsilon$ and $B=B_{st}+B_{\mathcal{I}}\epsilon$ are two matrices in $\mathbb{DQ}^{n\times n}$, where $A_{st},A_{\mathcal{I}},B_{st},B_{\mathcal{I}}\in\mathbb{Q}^{n\times n}$. Then $A$ and $B$ are unitarily similar in $\mathbb{DQ}^{n\times n}$ if and only if
\begin{enumerate}
\item $A_{st}$ and $B_{st}$ are unitarily similar in $\mathbb{Q}^{n\times n}$. That is, there exists a unitary matrix $U_{st}$ in $\mathbb{Q}^{n\times n}$ such that $U_{st}A_{st}U_{st}^{\ast}=B_{st}$.
\item For the same unitary matrix $U_{st}$, there exists an invertible matrix $P\in\mathbb{Q}^{2n\times 2n}$ such that
$$
P\begin{pmatrix}
B_{\mathcal{I}}-U_{st}A_{\mathcal{I}}U_{st}^{\ast} &B_{st}\\
B_{st} & 0
\end{pmatrix}P^{\ast}=\begin{pmatrix}
0       &B_{st}\\
B_{st} & 0
\end{pmatrix},
$$
$$
P\begin{pmatrix}
0 &I_{n}\\
I_{n} & 0
\end{pmatrix}P^{\ast}=\begin{pmatrix}
0       &I_{n}\\
I_{n} & 0
\end{pmatrix}.
$$
\end{enumerate}
\end{corollary}

\section{Trace-word criteria for unitary similarity in $\mathbb{DC}^{2\times 2}$}

In this section, we characterize the unitary similarity for matrices in $\mathbb{DC}^{2\times 2}$.
\begin{lemma}[Schur's Unitary Triangularization Theorem \cite{HR}]\label{1}
Suppose that $A\in\mathbb{C}^{n\times n}$, then $A$ is unitarily similar to a matrix of upper (or lower) triangular form.
\end{lemma}

In particular, $A$ is unitarily similar to a diagonal matrix if and only if $A$ is normal and in this case, the diagonal elements are precisely all eigenvalues of $A$. For matrices in $\mathbb{DC}^{n\times n}$, we have the following conclusion.

\begin{lemma}\label{5}
Suppose that $A=A_{st}+A_{\mathcal{I}}\epsilon\in\mathbb{DC}^{n\times n}$, where $A_{st},A_{\mathcal{I}}\in\mathbb{C}^{n\times n}$. Then $A$ is unitarily similar to a matrix $A^{\prime}=A_{st}^{\prime}+A_{\mathcal{I}}^{\prime}\epsilon$ in $\mathbb{DC}^{n\times n}$, where $A_{st}^{\prime}$, the standard part of $A^{\prime}$ is of the upper (or lower) triangular form.
\end{lemma}
\begin{proof}
Since $A_{st}\in\mathbb{C}^{n\times n}$, then by Lemma \ref{1}, there is a unitary matrix $U_{st}\in\mathbb{C}^{n\times n}$ such that
$$
U_{st}A_{st}U_{st}^{\ast}=A_{st}^{\prime},
$$
where $A_{st}^{\prime}\in \mathbb{C}^{n\times n}$ is of the upper (or lower) triangular form. By Lemma \ref{2}, the matrix equation
\begin{equation}\label{3}
U_{st}X+X^{\ast}U_{st}^{\ast}=0
\end{equation}
has a solution in $\mathbb{C}^{n\times n}$, since the matrix
$$
\begin{pmatrix}
0&U_{st}\\
U_{st}^{\ast}&0
\end{pmatrix}
$$
naturally has the required block form, the invertible matrix $P$ can be taken as the identity matrix $I_{2n}$ such that
$$
P\begin{pmatrix}
0&U_{st}\\
U_{st}^{\ast}&0
\end{pmatrix}P^{\ast}=\begin{pmatrix}
0&U_{st}\\
U_{st}^{\ast}&0
\end{pmatrix}.
$$
Let $U_{\mathcal{I}}=X^{\ast}$, then it follows from (\ref{3}) that $U_{st}U_{\mathcal{I}}^{\ast}$ is skew-Hermitian. Let $U=U_{st}+U_{\mathcal{I}}\epsilon$, then $U$ is a unitary matrix in $\mathbb{DC}^{n\times n}$ and
$$
\begin{aligned}
UAU^{\ast}&=(U_{st}+U_{\mathcal{I}}\epsilon)(A_{st}+A_{\mathcal{I}}\epsilon)(U_{st}^{\ast}+U_{\mathcal{I}}^{\ast}\epsilon)\\
          &=U_{st}A_{st}U_{st}^{\ast}+(U_{st}A_{st}U_{\mathcal{I}}^{\ast}+U_{st}A_{\mathcal{I}}U_{st}^{\ast}+U_{\mathcal{I}}A_{st}U_{st}^{\ast})\epsilon\\
          &=A_{st}^{\prime}+(U_{st}A_{st}U_{\mathcal{I}}^{\ast}+U_{st}A_{\mathcal{I}}U_{st}^{\ast}+U_{\mathcal{I}}A_{st}U_{st}^{\ast})\epsilon.
\end{aligned}
$$
Put $A_{\mathcal{I}}^{\prime}=U_{st}A_{st}U_{\mathcal{I}}^{\ast}+U_{st}A_{\mathcal{I}}U_{st}^{\ast}+U_{\mathcal{I}}A_{st}U_{st}^{\ast}$ and $A^{\prime}=A_{st}^{\prime}+A_{\mathcal{I}}^{\prime}\epsilon$, then we have
$$
UAU^{\ast}=A^{\prime}.
$$
\end{proof}

Now, let us define the trace for a matrix in $\mathbb{DC}^{n\times n}$.
\begin{definition}
Suppose that $A=A_{st}+A_{\mathcal{I}}\epsilon\in\mathbb{DC}^{n\times n}$, where $A_{st},A_{\mathcal{I}}\in\mathbb{C}^{n\times n}$. Define the trace of $A$ as
$$
\mathrm{tr}(A)=\mathrm{tr}(A_{st})+\mathrm{tr}(A_{\mathcal{I}})\epsilon.
$$
For matrices $M_1,\ldots,M_k$ whose trace is defined, we use the
componentwise notation
\[
\operatorname{tr}(M_1,\ldots,M_k)
:=
\bigl(\operatorname{tr}(M_1),\ldots,\operatorname{tr}(M_k)\bigr).
\]
Thus, equality between two such expressions means componentwise equality.
\end{definition}

Similar to the traces of complex matrices, the traces of dual complex matrices possess several properties.

\begin{proposition}\label{4}
Suppose that $A=A_{st}+A_{\mathcal{I}}\epsilon,B=B_{st}+B_{\mathcal{I}}\epsilon\in\mathbb{DC}^{n\times n}$, where $A_{st},A_{\mathcal{I}},B_{st},B_{\mathcal{I}}\in\mathbb{C}^{n\times n}$. $\lambda\in\mathbb{DC}$. Then
\begin{enumerate}
\item [(1)]$\mathrm{tr}(A+B)=\mathrm{tr}(A)+\mathrm{tr}(B)$.
\item [(2)]$\mathrm{tr}(AB)=\mathrm{tr}(BA)$.
\item [(3)]$\lambda \mathrm{tr}(A)=\mathrm{tr}(\lambda A)$.
\item [(4)]If $B=P^{-1}AP$ for some invertible matrix $P\in\mathbb{DC}^{n\times n}$, then $\mathrm{tr}(A)=\mathrm{tr}(B)$.
\item [(5)]$\mathrm{tr}(A^{T})=\mathrm{tr}(A)$.
\item [(6)]$\mathrm{tr}(\overline{A})=\overline{\mathrm{tr}(A)}$.
\item [(7)]$\mathrm{tr}(A^{\ast})=\overline{\mathrm{tr}(A)}$.
\end{enumerate}
\end{proposition}

In 1940, W. Specht \cite{WS40} established the
following necessary and sufficient criterion, now known
as Specht's criterion.
\begin{lemma}[\cite{WS40}]\label{sh}
Let $A,B\in\mathbb{C}^{n\times n}$, then $A$ is unitarily similar to $B$ if and only if 
\begin{equation}\label{sh1}
\mathrm{tr}\ W(A,A^{\ast})=\mathrm{tr}\ W(B,B^{\ast})
\end{equation}
for any word $W(s,t)$ in the non-commuting variables $s$ and $t$.
\end{lemma}

Lemma  \ref{sh} can be used to show that two complex matrices are not unitarily similar. It may be useless in showing that two given complex matrices are unitarily similar because infinitely many conditions must be verified. Fortunately, there is an improvement to Specht's result due to C. Pearcy  \cite{PC} which says it suffices to check the trace identities in (\ref{sh1}) for only finitely many words.
\begin{lemma}[\cite{PC}]
Two complex matrices $A,B$ are unitarily similar if and only if $\mathrm{tr}\ W(A,A^{\ast})=\mathrm{tr}\ W(B,B^{\ast})$ for every word $W(s,t)$ of degree at most $2n^{2}$.
\end{lemma}

In particular, F. Murnaghan \cite{FM} obtained the trace words for $2\times 2$ complex matrices.

\begin{lemma}[\cite{FM}]
Assume $A$ and $B$ are two $2\times 2$ complex matrices. Then $A$ is unitarily similar to $B$ if and only if $\mathrm{tr}(A,A^{2},A^{\ast}A)=\mathrm{tr}(B,B^{2},B^{\ast}B)$.
\end{lemma}

Inspired by F. Murnaghan's result, we discuss the role of trace words in characterizing unitary similarity for $2\times 2$ dual complex matrices $A$ and $B$ in cases. More precisely, it is divided into three cases: whether the standard part $A_{st}$ of $A$ is normal, and if $A_{st}$ is normal, whether the eigenvalues of $A_{st}$ are equal.

\begin{theorem}\label{ppp}
Suppose that $A=A_{st}+A_{\mathcal{I}}\epsilon,B=B_{st}+B_{\mathcal{I}}\epsilon\in\mathbb{DC}^{2\times 2}$ and $A_{st}$ is not normal. Then $A$ and $B$ are unitarily similar in $\mathbb{DC}^{2\times 2}$ if and only if
\begin{equation}\label{6}
\mathrm{tr}(A,A^{2},A^{\ast}A)=\mathrm{tr}(B,B^{2},B^{\ast}B).
\end{equation}
\end{theorem}
\begin{proof}
We prove the necessity first. Assume that there is a unitary matrix $U=U_{st}+U_{\mathcal{I}}\epsilon$ in $\mathbb{DC}^{2\times 2}$ such that $UAU^{\ast}=B$. This means
$$
UA^{2}U^{\ast}=B^{2},\ UA^{\ast}AU^{\ast}=B^{\ast}B.
$$
It follows from $(4)$ of Proposition \ref{4} that $\mathrm{tr}(A,A^{2},A^{\ast}A)=\mathrm{tr}(B,B^{2},B^{\ast}B)$. We can see that this statement is always true no matter whether $A$ or $B$ is normal or not and actually all the trace words of $A$ and $B$ are equal.

Now we prove the sufficiency. Assume that (\ref{6}) holds. By Lemma \ref{5}, we may assume that both the standard parts of $A$ and $B$ are of the upper triangular forms. For convenience, we let
$$
A=\begin{pmatrix}
a_{1}&a_{2}\\
0 &a_{4}
\end{pmatrix}+\begin{pmatrix}
b_{1}&b_{2}\\
b_{3} &b_{4}
\end{pmatrix}\epsilon,\ B=\begin{pmatrix}
c_{1}&c_{2}\\
0 &c_{4}
\end{pmatrix}+\begin{pmatrix}
d_{1}&d_{2}\\
d_{3} &d_{4}
\end{pmatrix}\epsilon.
$$

Since $A_{st}$ is not normal, it can be seen that $a_{2}$ is not equal to zero. From (\ref{6}) we obtain
\begin{equation}
\mathrm{tr}(A)=(a_{1}+a_{4})+(b_{1}+b_{4})\epsilon=(c_{1}+c_{4})+(d_{1}+d_{4})\epsilon=\mathrm{tr}(B)
\end{equation}
\begin{equation}
\begin{aligned}
\mathrm{tr}(A^{2})&=(a_{1}^{2}+a_{4}^{2})+2(a_{1}b_{1}+a_{2}b_{3}+a_{4}b_{4})\epsilon\\
         &=(c_{1}^{2}+c_{4}^{2})+2(c_{1}d_{1}+c_{2}d_{3}+c_{4}d_{4})\epsilon=\mathrm{tr}(B^{2})
\end{aligned}
\end{equation}
\begin{equation}
\begin{aligned}
\mathrm{tr}(A^{\ast}A)&=(|a_{1}|^{2}+|a_{2}|^{2}+|a_{4}|^{2})+2\mathrm{Re}(\overline{a_{1}}b_{1}+\overline{a_{2}}b_{2}+\overline{a_{4}}b_{4})\epsilon\\
             &=(|c_{1}|^{2}+|c_{2}|^{2}+|c_{4}|^{2})+2\mathrm{Re}(\overline{c_{1}}d_{1}+\overline{c_{2}}d_{2}+\overline{c_{4}}d_{4})\epsilon=\mathrm{tr}(B^{\ast}B)
\end{aligned}
\end{equation}
That is,
\begin{align}
&a_{1}+a_{4}=c_{1}+c_{4}\label{7}\\
&b_{1}+b_{4}=d_{1}+d_{4}\label{8}\\
&a_{1}^{2}+a_{4}^{2}=c_{1}^{2}+c_{4}^{2}\label{9}\\
&a_{1}b_{1}+a_{2}b_{3}+a_{4}b_{4}=c_{1}d_{1}+c_{2}d_{3}+c_{4}d_{4}\label{10}\\
&|a_{1}|^{2}+|a_{2}|^{2}+|a_{4}|^{2}=|c_{1}|^{2}+|c_{2}|^{2}+|c_{4}|^{2}\label{11}\\
&\mathrm{Re}(\overline{a_{1}}b_{1}+\overline{a_{2}}b_{2}+\overline{a_{4}}b_{4})=\mathrm{Re}(\overline{c_{1}}d_{1}+\overline{c_{2}}d_{2}+\overline{c_{4}}d_{4})\label{12}
\end{align}
By (\ref{7}) and (\ref{9}) we have
\begin{equation}\label{13}
a_{1}a_{4}=c_{1}c_{4}.
\end{equation}
Further, by (\ref{7}) and (\ref{13}) we deduce that $c_{1}$ and $c_{4}$ are solutions of the equation
$$
x^{2}-(a_{1}+a_{4})x+a_{1}a_{4}=0.
$$
Therefore
\begin{displaymath}
\left \{ \begin{array}{ll}
c_{1}=a_{1}\\
c_{4}=a_{4}
\end{array} \right.
 \text{or} \
\left \{ \begin{array}{ll}
c_{1}=a_{4}\\
c_{4}=a_{1}
\end{array} \right.
\end{displaymath}
Regardless of which set of solutions is taken, substituting it into (\ref{11}) gives $|a_{2}|=|c_{2}|$. Hence $c_{2}$ is nonzero. We may assume $c_{2}=a_{2}e^{\boldsymbol{i}\theta}$, where $0\leq\theta<2\pi$. Hence,
$$
B_{st}=\begin{pmatrix}
a_{1}&a_{2}e^{\boldsymbol{i}\theta}\\
0&a_{4}
\end{pmatrix}\  \text{or} \ B_{st}=\begin{pmatrix}
a_{4}&a_{2}e^{\boldsymbol{i}\theta}\\
0&a_{1}
\end{pmatrix}.
$$
In fact, no matter what form $B_{st}$ has, $B_{st}$ and $A_{st}$ are unitarily similar. If $B_{st}=\begin{pmatrix}
a_{1}&a_{2}e^{\boldsymbol{i}\theta}\\
0&a_{4}
\end{pmatrix},$ let $U_{1}=\begin{pmatrix}
1&0\\
0&e^{-\boldsymbol{i}\theta}
\end{pmatrix},$ then $U_{1}$ is unitary and $U_{1}A_{st}U_{1}^{\ast}=B_{st}$. If $B_{st}=\begin{pmatrix}
a_{4}&a_{2}e^{\boldsymbol{i}\theta}\\
0&a_{1}
\end{pmatrix},$ let $U_{1}=\begin{pmatrix}
u&v\\
-\overline{v}&\overline{u}
\end{pmatrix},$
where
$$
u=e^{\boldsymbol{i}\theta/2}|a_{2}|(|a_{1}-a_{4}|^{2}+|a_{2}|^{2})^{-\frac{1}{2}},$$
$$
 v=e^{\boldsymbol{i}\theta/2}(\overline{a_{4}}-\overline{a_{1}})|a_{2}|(\overline{a_{2}})^{-1}(|a_{1}-a_{4}|^{2}+|a_{2}|^{2})^{-\frac{1}{2}}.
$$ Then $U_{1}$ is unitary and $U_{1}A_{st}U_{1}^{\ast}=B_{st}$.

Therefore, after replacing $A$ by $U_1AU_1^*$ and relabelling the
entries of its infinitesimal part, we may assume that
\[
A_{st}=B_{st}=
\begin{pmatrix}
a_1&a_2\\
0&a_4
\end{pmatrix}.
\]
The trace identities (\ref{8}), (\ref{10}), and (\ref{12}) remain valid because
the trace is invariant under unitary similarity. 

Let $U_{st}=I_2$. By Theorem \ref{14}, it is sufficient to prove that
the matrix equation
\begin{equation}\label{4.14}
B_{st}X-XB_{st}=B_{\mathcal{I}}-A_{\mathcal{I}}                                     
\end{equation}
has a skew-Hermitian solution $X\in\mathbb{C}^{2\times2}$. Write
\[
X=
\begin{pmatrix}
x_1&x_2\\
x_3&x_4
\end{pmatrix}.
\]
The condition $X^*=-X$ is equivalent to
\begin{numcases}{}
\overline{x_1}+x_1=0,  \label{4.15}\\
\overline{x_3}+x_2=0,   \label{4.16}\\
\overline{x_4}+x_4=0,    \label{4.17}\\
\overline{x_2}+x_3=0.  \label{4.18}
\end{numcases}{}

A direct computation gives
\[
B_{st}X-XB_{st}
=
\begin{pmatrix}
a_2x_3&
(a_1-a_4)x_2+a_2(x_4-x_1)\\
(a_4-a_1)x_3&
-a_2x_3
\end{pmatrix}.
\]
Hence, equation (\ref{4.14}) is equivalent to

\begin{numcases}{}
a_2x_3=d_1-b_1, \label{4.19}\\
-a_2x_3=d_4-b_4,  \label{4.20}\\
(a_1-a_4)x_2+a_2(x_4-x_1)=d_2-b_2,  \label{4.21}\\
(a_4-a_1)x_3=d_3-b_3.\label{4.22}
\end{numcases}{}

It follows from (\ref{8}) that
\[
d_4-b_4=-(d_1-b_1).
\]
Moreover, (\ref{8}) and (\ref{10}) imply
\[
a_2(d_3-b_3)=(a_4-a_1)(d_1-b_1).
\]
We now set
\[
x_3=\frac{d_1-b_1}{a_2},
\qquad
x_2=-\overline{x_3}.
\]
Then (\ref{4.16}), (\ref{4.18}), (\ref{4.19}), (\ref{4.20}) and (\ref{4.22}) are satisfied.

Define
\[
z=
\frac{d_2-b_2-(a_1-a_4)x_2}{a_2}.                          
\]
We claim that $z$ is purely imaginary. Indeed, by (\ref{8}) and
(\ref{12}),
\[
\operatorname{Re}\!\left(
(\overline{a_1}-\overline{a_4})(d_1-b_1)
+\overline{a_2}(d_2-b_2)
\right)=0.
\]
Since
\[
x_2=-\frac{\overline{d_1-b_1}}{\overline{a_2}},
\]
we have
\[
\begin{aligned}
|a_2|^2\operatorname{Re}(z)
&=
\operatorname{Re}\!\left(
\overline{a_2}(d_2-b_2)
+(a_1-a_4)\overline{(d_1-b_1)}
\right)\\
&=
\operatorname{Re}\!\left(
\overline{a_2}(d_2-b_2)
+(\overline{a_1}-\overline{a_4})(d_1-b_1)
\right)\\
&=0.
\end{aligned}
\]
Thus $\operatorname{Re}(z)=0$.

Finally, take
\[
x_1=0,\qquad x_4=z.
\]
Then (\ref{4.15}) and (\ref{4.17}) hold, and the definition of $z$ shows that
(\ref{4.21}) also holds. Therefore, $X$ is a skew-Hermitian solution of
(\ref{4.14}).

Let
\[
V=I_2+X^*\epsilon.
\]
By Theorem \ref{14}, $V$ is unitary and
\[
V(U_1AU_1^*)V^*=B.
\]
Consequently,
\[
U=VU_1=U_1+X^*U_1\epsilon
\]
is unitary and satisfies
\[
UAU^*=B.
\]
Therefore, $A$ and $B$ are unitarily similar.
\end{proof}

When $A_{st}$ is normal, we still have the following conclusions.

\begin{proposition}\label{j}
Suppose that $A=A_{st}+A_{\mathcal{I}}\epsilon,B=B_{st}+B_{\mathcal{I}}\epsilon\in\mathbb{DC}^{2\times 2}$ and $A_{st}$ is normal. Assume further that the two eigenvalues of $A_{\rm st}$ are equal.  Then $A$ is unitarily similar to $B$ if and only if the ordered pairs $\{A_{st},A_{\mathcal{I}}\}$ and $\{B_{st},B_{\mathcal{I}}\}$ are simultaneously unitarily similar, i.e., there exists a unitary matrix $U_{st}$ such that $U_{st}A_{st}U_{st}^{\ast}=B_{st}$ and $U_{st}A_{\mathcal{I}}U_{st}^{\ast}=B_{\mathcal{I}}$.
\end{proposition}
\begin{proof}
Since the two eigenvalues of $A_{\rm st}$ are equal, we may assume,
up to unitary similarity, that $A_{\rm st}=a_1I_2$. First we prove the necessity.

Suppose that $A$ is unitarily similar to $B$. By Theorem \ref{14},
\begin{enumerate}
\item There exists a unitary matrix $U_{st}$ in $\mathbb{C}^{2\times 2}$ such that
\begin{equation}\label{p}
U_{st}A_{st}U_{st}^{\ast}=B_{st}.
\end{equation}
Notice that $A_{st}$ is a scalar matrix, hence
$$
B_{st}=U_{st}A_{st}U_{st}^{\ast}=A_{st}U_{st}U_{st}^{\ast}=A_{st}.
$$
In fact, for any unitary matrix $U_{st}\in\mathbb{C}^{2\times 2}$, we always have $B_{st}=A_{st}$.
\item For the same unitary matrix $U_{st}$, 
 the matrix equation (\ref{15}) has a skew-Hermitian solution $X\in\mathbb{C}^{2\times 2}$. It follows from the form of $B_{st}$ that
    $$
    0_{2}=B_{st}X-XB_{st}=B_{\mathcal{I}}-U_{st}A_{\mathcal{I}}U_{st}^{\ast}.
    $$
    Therefore,
\begin{equation}\label{pp}
U_{st}A_{\mathcal{I}}U_{st}^{\ast}=B_{\mathcal{I}}.
\end{equation}
That is, the ordered pairs $\{A_{st},A_{\mathcal{I}}\}$ and $\{B_{st},B_{\mathcal{I}}\}$ are simultaneously unitarily similar.
\end{enumerate}

Now we prove the sufficiency. Assume that the unitary matrix $U_{st}\in\mathbb{C}^{2\times 2}$ satisfies both (\ref{p}) and (\ref{pp}). By (\ref{p}) and the form of $A_{st}$, we can deduce that $B_{st}=A_{st}$ is also scalar. Thus the left side of equation (\ref{15}) is a zero matrix for any $X\in\mathbb{C}^{2\times 2}$, which is equal to the right side of (\ref{15}). This shows that (\ref{15}) has a skew-Hermitian solution $X\in\mathbb{C}^{2\times 2}$. Finally, by Theorem \ref{14}, $A$ is unitarily similar to $B$.
\end{proof}

It is worth noting that if $A=A_{st}+A_{\mathcal{I}}\epsilon,B=B_{st}+B_{\mathcal{I}}\epsilon\in\mathbb{DC}^{2\times 2}$ and $A_{st}$ is normal, the eigenvalues of $A_{st}$ are equal, then
$$
\mathrm{tr}(A,A^{2},A^{\ast}A)=\mathrm{tr}(B,B^{2},B^{\ast}B)
$$
is not enough to guarantee the unitary similarity between $A$ and $B$. In fact, it is possible that even if all the trace words of $A$ and $B$ are equal, they may still not be unitarily similar.

\begin{example}
Let $A=\begin{pmatrix}
1&0\\
0&1
\end{pmatrix}+\begin{pmatrix}
1&x\\
y&1
\end{pmatrix}\epsilon
$,                    $B=\begin{pmatrix}
1&0\\
0&1
\end{pmatrix}+\begin{pmatrix}
1&0\\
0&1
\end{pmatrix}\epsilon\in\mathbb{DC}^{2\times 2}$, where $x,y\in\mathbb{C}$ are not both zero. 
For any binary word $W(a, b)$, where $a$ appears $m$ times and $b$ appears $n$ times, one can see that
\[
W(A,A^{\ast})=\begin{pmatrix}
	1&0\\
	0&1
\end{pmatrix}+\begin{pmatrix}
	m+n&mx+n\bar{x}\\
	my+n\bar{y}&m+n
\end{pmatrix}\epsilon \ \ \text{and} \ \ 
W(B,B^{\ast})=\begin{pmatrix}
	1&0\\
	0&1
\end{pmatrix}+\begin{pmatrix}
	m+n&0\\
	0&m+n
\end{pmatrix}\epsilon.
\]
Consequently,
\[
\mathrm{Tr}(W(A,A^{\ast}))=2+2(m+n)\epsilon=\mathrm{Tr}(W(B,B^{\ast})).
\]
Therefore, all the trace words of $A$ and $B$ are equal.
However, $A$ is not unitarily similar to $B$. Since $A_{st}=B_{st}$ is scalar and $A_{\mathcal{I}}$ is not unitarily similar to $B_{\mathcal{I}}$, it follows from Proposition \ref{j} that $A$ and $B$ are not unitarily similar.
\end{example}
\begin{remark}\label{55}
From the proof of Theorem \ref{ppp} and Proposition \ref{j} we identify the obstruction in the case where $A_{st}$ is normal and has two distinct eigenvalues. In this case, we have $a_{2}=0$ and $a_{1}\neq a_{4}$. Further, by the trace identities (\ref{8}) and (\ref{10}), we obtain $b_{1}=d_{1}$ and $b_{4}=d_{4}$. Let $U_{st}\in\mathbb{C}^{2\times 2}$ be the unitary matrix such that
$$
U_{st}A_{st}U_{st}^{\ast}=B_{st}=\begin{pmatrix}
a_{1}&0\\
0&a_{4}
\end{pmatrix}=A_{st}.
$$
We deduce that $U_{st}$ must be of the form
$$
U_{st}=\begin{pmatrix}
e^{\boldsymbol{i}\theta}&0\\
0&e^{\boldsymbol{i}\beta}
\end{pmatrix},\ \text{for} \ 0<\theta,\beta\leq 2\pi.
$$

Let
\[
X=
\begin{pmatrix}
x_1&x_2\\
x_3&x_4
\end{pmatrix}
\in\mathbb{C}^{2\times2}.
\]
By Theorem \ref{14}, it remains to determine whether the equation
\[
B_{st}X-XB_{st}
=
B_{\mathcal{I}}-U_{st}A_{\mathcal{I}}U_{st}^*
\]
has a skew-Hermitian solution. Since
\[
U_{st}=
\begin{pmatrix}
e^{\boldsymbol{i}\theta}&0\\
0&e^{\boldsymbol{i}\beta}
\end{pmatrix},
\]
we have
\[
B_{\mathcal{I}}-U_{st}A_{\mathcal{I}}U_{st}^*
=
\begin{pmatrix}
d_1-b_1&
d_2-e^{\boldsymbol{i}(\theta-\beta)}b_2\\
d_3-e^{\boldsymbol{i}(\beta-\theta)}b_3&
d_4-b_4
\end{pmatrix}.
\]
By the trace identities (\ref{8}) and (\ref{10}),
\[
d_1=b_1,\qquad d_4=b_4.
\]
On the other hand,
\[
B_{st}X-XB_{st}
=
\begin{pmatrix}
0&(a_1-a_4)x_2\\
(a_4-a_1)x_3&0
\end{pmatrix}.
\]
Therefore, the equation is equivalent to
\begin{align*}
(a_1-a_4)x_2
&=d_2-e^{\boldsymbol{i}(\theta-\beta)}b_2,\\
(a_4-a_1)x_3
&=d_3-e^{\boldsymbol{i}(\beta-\theta)}b_3.
\end{align*}

Since $X$ is required to be skew-Hermitian, we must have
\[
x_3=-\overline{x_2}.
\]
Consequently, the above equation has a skew-Hermitian solution if
and only if
\begin{equation}\label{4.25}
\frac{d_3-e^{\boldsymbol{i}(\beta-\theta)}b_3}{a_1-a_4}
=
\frac{
\overline{d_2-e^{\boldsymbol{i}(\theta-\beta)}b_2}
}{
\overline{a_1-a_4}
}.                                                          
\end{equation}
Indeed, when (\ref{4.25}) holds, we may take
\[
x_2=
\frac{d_2-e^{\boldsymbol{i}(\theta-\beta)}b_2}{a_1-a_4},
\qquad
x_3=-\overline{x_2},
\qquad
x_1=x_4=0.
\]
Then $X$ is skew-Hermitian and satisfies the required matrix
equation.

Nevertheless, in the situation under consideration, condition
(\ref{4.25}) does not always hold. Hence
\[
\operatorname{tr}(A,A^2,A^*A)
=
\operatorname{tr}(B,B^2,B^*B)
\]
is not sufficient to guarantee that $A$ and $B$ are unitarily similar.

It is worth noting that if
\begin{equation}\label{4.26}
a_4=\overline{a_1},\qquad
b_3=-\overline{b_2},\qquad
d_3=-\overline{d_2},                                     
\end{equation}
then condition (\ref{4.25}) holds automatically. In fact, if
\[
C_{12}=d_2-e^{\boldsymbol{i}(\theta-\beta)}b_2,
\]
then
\[
d_3-e^{\boldsymbol{i}(\beta-\theta)}b_3=-\overline{C_{12}},
\qquad
\overline{a_1-a_4}=-(a_1-a_4).
\]
Thus the two sides of (\ref{4.25}) are equal. In particular, choosing
\[
e^{\boldsymbol{i}\beta}=e^{-\boldsymbol{i}\theta}
\]
makes $U_{st}$ have complex adjoint form.

\end{remark}

For Remark \ref{55}, we give a concrete example.
\begin{example}
Assume $A=\begin{pmatrix}
\boldsymbol{i}&0\\
0&-\boldsymbol{i}
\end{pmatrix}+\begin{pmatrix}
1&1\\
1&1
\end{pmatrix}\epsilon
$,                    $B=\begin{pmatrix}
\boldsymbol{i}&0\\
0&-\boldsymbol{i}
\end{pmatrix}+\begin{pmatrix}
1&1\\
2&1
\end{pmatrix}\epsilon\in\mathbb{DC}^{2\times 2}$. Then $\mathrm{tr}(A)=\mathrm{tr}(B)=2\epsilon$ and
$$
A^{2}=B^{2}=\begin{pmatrix}
-1&0\\
0&-1
\end{pmatrix}+\begin{pmatrix}
2\boldsymbol{i}&0\\
0&-2\boldsymbol{i}
\end{pmatrix}\epsilon,
$$
$$
 A^{\ast}A=\begin{pmatrix}
1&0\\
0&1
\end{pmatrix}+\begin{pmatrix}
0&-2\boldsymbol{i}\\
2\boldsymbol{i}&0
\end{pmatrix}\epsilon,\  B^{\ast}B=\begin{pmatrix}
1&0\\
0&1
\end{pmatrix}+\begin{pmatrix}
0&-3\boldsymbol{i}\\
3\boldsymbol{i}&0
\end{pmatrix}\epsilon.
$$
 Therefore,
$$
\mathrm{tr}(A^{2})=\mathrm{tr}(B^{2})=-2,\ \mathrm{tr}(A^{\ast}A)=\mathrm{tr}(B^{\ast}B)=2
$$
and
$$
\mathrm{tr}(A,A^{2},A^{\ast}A)=\mathrm{tr}(B,B^{2},B^{\ast}B).
$$

In order for \(A\) and \(B\) to be unitarily similar, it is
necessary that Equation~\eqref{4.25} holds. Let
\[
\phi=\theta-\beta.
\]
For the matrices in this example, we have
\[
a_1-a_4=2\boldsymbol{i},\qquad
d_2-e^{\boldsymbol{i}(\theta-\beta)}b_2=1-e^{\boldsymbol{i}\phi},
\]
and
\[
d_3-e^{\boldsymbol{i}(\beta-\theta)}b_3=2-e^{-\boldsymbol{i}\phi}.
\]
Therefore, Equation~\eqref{4.25} becomes
\[
\frac{2-e^{-\boldsymbol{i}\phi}}{2i}
=
\frac{\overline{1-e^{\boldsymbol{i}\phi}}}{\overline{2\boldsymbol{i}}}
=
\frac{1-e^{-\boldsymbol{i}\phi}}{-2\boldsymbol{i}}.
\]
Hence
\[
2-e^{-\boldsymbol{i}\phi}=-1+e^{-\boldsymbol{i}\phi},
\]
or equivalently,
\[
3=2e^{-\boldsymbol{i}\phi}.
\]
This is impossible because \(\lvert e^{-\boldsymbol{i}\phi}\rvert=1\).
Therefore, the equation system has no solution, which implies that
\(A\) and \(B\) cannot be unitarily similar.

In fact, all trace words of $A$ and $B$ coincide. Given any binary word $W(a, b)$, where $a$ appears $m$ times and $b$ appears $n$ times. Since $A_{st}=B_{st}$ are diagonal matrices and the diagonals of matrices $A_{\mathcal{I}}$ and $B_{\mathcal{I}}$ are the same, one can see that
$$
\mathrm{tr}\ W(A,A^{\ast})=\mathrm{tr}\ W(B,B^{\ast}).
$$
This example further demonstrates that equality of all trace words does not guarantee the unitary similarity of $A$ and $B$.
\end{example}

\section{Complex-adjoint criteria for unitary similarity in $\mathbb{DQ}^{n\times n}$}
In this section, we first establish the relation between unitary similarity
of arbitrary-order dual quaternion matrices and that of their dual complex
adjoint matrices.  The scalar Trace-word results are then obtained as
consequences.

We retain the following observation for the scalar real-standard-part case.

\begin{proposition}\label{99}
Suppose that $A=A_{st}+A_{\mathcal{I}}\epsilon,B=B_{st}+B_{\mathcal{I}}\epsilon\in\mathbb{DQ}$ and $A_{st}$ is real. Then $A$ is unitarily similar to $B$ if and only if
the ordered pairs $\{A_{st},A_{\mathcal{I}}\}$ and $\{B_{st},B_{\mathcal{I}}\}$ are simultaneously unitarily similar.
\end{proposition}

\begin{proof}
The proof is analogous to that of Proposition \ref{j}.
\end{proof}

N. Wiegmann \cite{WN} proved that two $n\times n$ quaternion matrices are unitarily similar if and only if their complex adjoint matrices are unitarily similar. 

\begin{lemma}[\cite{WN}]\label{hu}
Assume that $A,B\in\mathbb{Q}^{n\times n}$. Then $A$ is unitarily similar to $B$ in $\mathbb{Q}^{n\times n}$ if and only if $A_{\mathbb{C}}$ is unitarily similar to $B_{\mathbb{C}}$ in $\mathbb{C}^{2n\times 2n}$.
\end{lemma}

For the proof in the dual setting, we need a matrix form of the symmetry
that characterizes complex adjoint matrices. Put
\[
 J=\begin{pmatrix}0&I_n\\-I_n&0\end{pmatrix},
 \qquad
 \Theta(X)=J\overline XJ^{-1}
 \quad (X\in\mathbb{C}^{2n\times2n}).
\]
The map $\Theta$ is real-linear, conjugate-linear over $\mathbb{C}$,
multiplicative, adjoint-preserving, and satisfies
$\Theta^2=\operatorname{id}$. A direct
comparison of the four $n\times n$ blocks gives
\begin{equation}\label{eq:fixed-adjoint}
 Fix(\Theta)
 =\{X\in\mathbb{C}^{2n\times2n}:\Theta(X)=X\}
 =\{A_{\mathbb C}:A\in\mathbb{Q}^{n\times n}\}.
\end{equation}
In particular, a unitary matrix fixed by $\Theta$ is the complex adjoint of
a quaternion unitary matrix. We extend $\Theta$ coefficientwise to dual
complex matrices.

The next lemma is the fixed-point formulation of the simultaneous extension
recorded by Wiegmann \cite[Section~3]{WN}. We include a direct
proof because the construction will also be used in the dual setting.

\begin{lemma}\label{lem:replacement}
Let $C_1,\ldots,C_s,D_1,\ldots,D_s\in Fix(\Theta)$. If a complex unitary
matrix $V$ satisfies
\[
 D_r=VC_rV^*,\qquad r=1,\ldots,s,
\]
then there is a complex unitary matrix $Q$, fixed by $\Theta$, satisfying
\[
 D_r=QC_rQ^*,\qquad r=1,\ldots,s.
\]
\end{lemma}

\begin{proof}
The given equations are equivalent to $VC_r=D_rV$. Applying $\Theta$ and
using $\Theta(C_r)=C_r$ and $\Theta(D_r)=D_r$ shows that $\Theta(V)$ obeys
the same equations. Define
\[
 R=\frac{V+\Theta(V)}2,
 \qquad
 S=\frac{V-\Theta(V)}{2\boldsymbol{i}}.
\]
Because $\Theta$ is conjugate-linear, both $R$ and $S$ are fixed by
$\Theta$. They simultaneously intertwine all $C_r,D_r$. They also
intertwine the adjoint pairs. Indeed, taking adjoints in
$D_r=VC_rV^*$ gives $D_r^*=VC_r^*V^*$, and hence
$VC_r^*=D_r^*V$. Since $C_r^*$ and $D_r^*$ are fixed by $\Theta$, the
same averaging argument gives
\[
 RC_r^*=D_r^*R,\qquad SC_r^*=D_r^*S.
\]

Regard
\[
 p(z)=\det(R+zS),\qquad z\in\mathbb{C},
\]
as a complex polynomial. Since
$p(\boldsymbol{i})=\det(R+\boldsymbol{i}S)=\det V\ne0$, it is not identically zero. Consequently,
it has only finitely many real zeros, so some $t_0\in\mathbb{R}$ makes
$P=R+t_0S$ invertible. Then $\Theta(P)=P$ and
\[
 PC_r=D_rP,\qquad PC_r^*=D_r^*P.
\]
Taking the adjoint of the second equality and using the first one yields
\[
 P^*PC_r=P^*D_rP=C_rP^*P.
\]
Let $H=(P^*P)^{1/2}$ be the unique positive definite square root
\cite{HR}. Then $H$ and $H^{-1}$ commute with every $C_r$.
Moreover, $\Theta(P^*P)=P^*P$. Since $J$ is unitary, $\Theta$ maps
positive definite matrices to positive definite matrices, and
\[
 \Theta(H)^2=\Theta(H^2)=\Theta(P^*P)=P^*P.
\]
Thus $H$ and $\Theta(H)$ are both positive definite square roots of
$P^*P$. Uniqueness of the positive definite square root gives
$\Theta(H)=H$. Therefore
\[
 Q=PH^{-1}
\]
is fixed by $\Theta$. Since $H$ is positive definite, it is Hermitian and
invertible. Hence
\[
 Q^*Q=(PH^{-1})^*(PH^{-1})
      =H^{-1}P^*PH^{-1}
      =H^{-1}H^2H^{-1}=I_{2n}.
\]
Thus $Q$ is unitary. Moreover, it satisfies $QC_r=D_rQ$ for every $r$.
Consequently, $D_r=QC_rQ^*$.
\end{proof}

We next isolate the part of a dual coefficient that can be altered by a
first-order unitary correction. Write $d=2n$, regard
$\mathbb{C}^{d\times d}$ as a real vector space, and use the real Frobenius inner
product
\begin{equation}\label{eq:real-inner-product}
 \langle X,Y\rangle_{\mathbb{R}}=Re~ tr(X^*Y).
\end{equation}
This real inner product was also used by Wu et al.~\cite{WuLiZhangDuan2011}
to formulate orthogonality relations and to prove finite-step convergence,
in exact arithmetic, of an iterative method for coupled
Sylvester--conjugate matrix equations.
It is positive definite because
$\langle X,X\rangle_{\mathbb{R}}=\sum_{r,s}|x_{rs}|^2$. Put
\[
 \mathcal{K}_d=\{K\in\mathbb{C}^{d\times d}:K^*=-K\},\ \mathcal{R}_C=\{KC-CK:K\in\mathcal{K}_d\},\ \mathcal{E}_C=\mathcal{R}_C^\perp.
\]
All these are real vector spaces, and finite-dimensional orthogonal
decomposition gives
\begin{equation}\label{eq:decomposition}
 \mathbb{C}^{d\times d}=\mathcal{R}_C\oplus\mathcal{E}_C.
\end{equation}
Let $\pi_C$ be the orthogonal projection onto $\mathcal{E}_C$.

\begin{lemma}\label{lem:projection}
Assume $\Theta(C)=C$. Then the following statements hold.
\begin{enumerate}
\item[(i)] Both $\mathcal{R}_C$ and $\mathcal{E}_C$ are preserved by $\Theta$, and
\[
 \pi_C(\Theta X)=\Theta(\pi_CX).
\]
\item[(ii)]  If $g\in \mathbb{C}^{d\times d}$ is unitary and $gCg^*=C$, then conjugation by $g$ preserves
$\mathcal{R}_C$ and $\mathcal{E}_C$, and for any $X\in \mathbb{C}^{d\times d}$, we have
\[
 \pi_C(gXg^*)=g\pi_C(X)g^*.
\]
\item[(iii)]  If $D\in\mathcal{R}_C$ and $\Theta(D)=D$, then there is a matrix $K_0$
such that
\[
 K_0^*=-K_0,\qquad \Theta(K_0)=K_0,\qquad D=K_0C-CK_0.
\]
\end{enumerate}
\end{lemma}

\begin{proof}
If $K^*=-K$, then $\Theta(K)^*=-\Theta(K)$ and
\[
 \Theta(KC-CK)=\Theta(K)C-C\Theta(K).
\]
Thus $\Theta$ maps $\mathcal{R}_C$ onto itself. It is an isometry for
(\ref{eq:real-inner-product}); we give the details. For every
$Z\in\mathbb{C}^{d\times d}$, cyclic invariance of the trace gives
\[
 tr(\Theta(Z))
 =tr(J\overline ZJ^{-1})
 =tr(\overline Z)
 =\overline{tr(Z)}.
\]
Since $\Theta$ is multiplicative and adjoint-preserving, for all
$X,Y\in\mathbb{C}^{d\times d}$ we obtain
\begin{align*}
 \langle\Theta(X),\Theta(Y)\rangle_{\mathbb{R}}
 &=Re~ tr\bigl(\Theta(X)^*\Theta(Y)\bigr)\\
 &=Re~ tr\bigl(\Theta(X^*Y)\bigr)\\
 &=Re~\overline{tr(X^*Y)}
  =\langle X,Y\rangle_{\mathbb{R}}.
\end{align*}
Thus $\Theta$ is a real-linear isometry. To see explicitly that it preserves
$\mathcal{E}_C=\mathcal{R}_C^\perp$, let $X\in\mathcal{E}_C$ and $R\in\mathcal{R}_C$. Since
$\Theta(\mathcal{R}_C)=\mathcal{R}_C$ and $\Theta^2=\operatorname{id}$, there is
$R_0\in\mathcal{R}_C$ such that $R=\Theta(R_0)$. Hence
\[
 \langle\Theta(X),R\rangle_{\mathbb{R}}
 =\langle\Theta(X),\Theta(R_0)\rangle_{\mathbb{R}}
 =\langle X,R_0\rangle_{\mathbb{R}}=0.
\]
Therefore $\Theta(X)\in\mathcal{E}_C$. Applying $\Theta$ once more shows
$\Theta(\mathcal{E}_C)=\mathcal{E}_C$. Finally, if
$X=X_R+X_E$ is the decomposition in (\ref{eq:decomposition}), then
\[
 \Theta(X)=\Theta(X_R)+\Theta(X_E),
 \qquad
 \Theta(X_R)\in\mathcal{R}_C,\qquad \Theta(X_E)\in\mathcal{E}_C.
\]
Uniqueness of the orthogonal decomposition yields
$\pi_C(\Theta X)=\Theta(\pi_CX)$ and proves (i).

If $gCg^*=C$, then
\[
 g(KC-CK)g^*=(gKg^*)C-C(gKg^*),
\]
and $gKg^*$ is skew-Hermitian. Hence the map
$T_g(X)=gXg^*$ maps $\mathcal{R}_C$ onto itself; its inverse is $T_{g^*}$.
Moreover, unitary conjugation is an isometry for
(\ref{eq:real-inner-product}), because
\begin{align*}
 \langle T_g(X),T_g(Y)\rangle_{\mathbb{R}}
 &=Re~ tr\bigl((gXg^*)^*gYg^*\bigr)\\
 &=Re~ tr\bigl(gX^*Yg^*\bigr)\\
 &=Re~ tr(X^*Y)
  =\langle X,Y\rangle_{\mathbb{R}}.
\end{align*}
If $X\in\mathcal{E}_C$ and $R\in\mathcal{R}_C$, write $R=T_g(R_0)$ with
$R_0\in\mathcal{R}_C$. Then
\[
 \langle T_g(X),R\rangle_{\mathbb{R}}
 =\langle T_g(X),T_g(R_0)\rangle_{\mathbb{R}}
 =\langle X,R_0\rangle_{\mathbb{R}}=0.
\]
Thus $T_g(\mathcal{E}_C)=\mathcal{E}_C$. Applying $T_g$ to the decomposition
$X=X_R+X_E$ and using its uniqueness gives
\[
 \pi_C(gXg^*)=g\pi_C(X)g^*,
\]
which proves (ii).

For (iii), write $D=KC-CK$ with $K^*=-K$. Then
\[
 K_0=\frac{K+\Theta(K)}2
\]
is skew-Hermitian and fixed by $\Theta$, while
\[
 K_0C-CK_0
 =\frac12\bigl(D+\Theta(D)\bigr)=D.
\]
\end{proof}

We can now extend the connection in Lemma~\ref{hu} to dual
quaternion matrices of every order.

\begin{theorem}[Complex adjoint criterion in arbitrary order]\label{thm:arbitrary-order}
Let $A,B\in\mathbb{DQ}^{n\times n}$, where $n\ge1$. Then
\begin{equation}\label{eq:main-equivalence}
 A\sim_u B\ \text{in }\mathbb{DQ}^{n\times n}
 \quad\Longleftrightarrow\quad
 A_{\mathbb C}\sim_u B_{\mathbb C}\ \text{in }\mathbb{DC}^{2n\times2n}.
\end{equation}
More precisely, if the equivalence on the right is realized by a dual complex unitary matrix, then it is also realized by a
possibly different dual complex unitary matrix fixed by $\Theta$; equivalently, this matrix has the form $W_{\mathbb C}$
for a dual quaternion unitary matrix $W$.
\end{theorem}

\begin{proof}
Write
\[
 A=A_{st}+A_{\mathcal{I}}\epsilon,
 \qquad
 B=B_{st}+B_{\mathcal{I}}\epsilon,
\]
and abbreviate
\[
 C_0=(A_{st})_{\mathbb C},\quad C_1=(A_{\mathcal{I}})_{\mathbb C},
 \qquad
 D_0=(B_{st})_{\mathbb C},\quad D_1=(B_{\mathcal{I}})_{\mathbb C}.
\]
All four matrices are fixed by $\Theta$.

If $B=WAW^*$ for a dual quaternion unitary matrix $W$, then the algebraic
properties of the dual complex adjoint give
$B_{\mathbb C}=W_{\mathbb C}A_{\mathbb C}W_{\mathbb C}^*$, and
$W_{\mathbb C}$ is dual complex unitary. This proves the
forward implication.

Conversely, suppose
\begin{equation}\label{eq:given-dual-conjugacy}
 D_0+D_1\epsilon
 =\widehat U(C_0+C_1\epsilon)\widehat U^*
\end{equation}
for a dual complex unitary matrix
$\widehat U=U_0+U_1\epsilon$. Comparing coefficients in
$\widehat U\widehat U^*=I_{2n}$ shows that $U_0$ is unitary and
\[
 K:=U_1U_0^*\quad\text{satisfies}\quad K^*=-K.
\]
Equivalently, $\widehat U=(I_{2n}+K\epsilon)U_0$. Comparing the two
coefficients in \eqref{eq:given-dual-conjugacy} gives
\begin{align}
 D_0&=U_0C_0U_0^*,\label{eq:standard-coefficient}\\
 D_1&=U_0C_1U_0^*+(KD_0-D_0K).
 \label{eq:dual-coefficient}
\end{align}

Apply Lemma~\ref{lem:replacement} to the single pair $(C_0,D_0)$. There is
a unitary matrix $S$, fixed by $\Theta$, such that
\[
 D_0=SC_0S^*.
\]
Set
\[
 C=D_0,\qquad X=SC_1S^*,\qquad Y=D_1,\qquad g=U_0S^*.
\]
Then $C,X,Y$ are fixed by $\Theta$, $gCg^*=C$, and
\begin{equation}\label{eq:reduced-dual-equation}
 Y=gXg^*+(KC-CK).
\end{equation}
Applying $\pi_C$ and using Lemma~\ref{lem:projection}(ii) gives
\begin{equation}\label{eq:projected-equation}
 Y_E=gX_Eg^*,\qquad X_E:=\pi_C(X),\quad Y_E:=\pi_C(Y).
\end{equation}
Lemma~\ref{lem:projection}(i) shows that $X_E$ and $Y_E$ are fixed by
$\Theta$.

The unitary $g$ simultaneously conjugates $(C,X_E)$ to $(C,Y_E)$.
Lemma~\ref{lem:replacement}, now applied to this pair of matrices, provides
a unitary $h$, fixed by $\Theta$, such that
\begin{equation}\label{eq:structured-pair}
 hCh^*=C,\qquad hX_Eh^*=Y_E.
\end{equation}
Let
\[
 \Delta=Y-hXh^*.
\]
It is fixed by $\Theta$. By Lemma~\ref{lem:projection}(ii) and
\eqref{eq:structured-pair},
\[
 \pi_C(\Delta)=Y_E-hX_Eh^*=0.
\]
Therefore $\Delta\in\ker\pi_C=\mathcal{R}_C$. Part (iii) of
Lemma~\ref{lem:projection} yields a matrix $K_0$ such that
\[
 K_0^*=-K_0,\qquad \Theta(K_0)=K_0,\qquad
 \Delta=K_0C-CK_0.
\]

Define
\[
 \widehat T=(I_{2n}+K_0\epsilon)h,
 \qquad
 \widehat W=\widehat T S.
\]
We verify explicitly both the unitarity of $\widehat T$ and the effect of
conjugation by $\widehat T$. Since $\epsilon$ is real and commutes with
complex matrices, $K_0^*=-K_0$ gives
\[
 \widehat T^*
 =\bigl((I_{2n}+K_0\epsilon)h\bigr)^*
 =h^*(I_{2n}+K_0^*\epsilon)
 =h^*(I_{2n}-K_0\epsilon).
\]
Because $h$ is unitary and $\epsilon^2=0$, it follows that
\begin{align*}
 \widehat T\widehat T^*
 &=(I_{2n}+K_0\epsilon)hh^*(I_{2n}-K_0\epsilon)\\
 &=(I_{2n}+K_0\epsilon)(I_{2n}-K_0\epsilon)
   =I_{2n}-K_0^2\epsilon^2=I_{2n},
\end{align*}
\begin{samepage}
and, similarly,
\begin{align*}
 \widehat T^*\widehat T
 &=h^*(I_{2n}-K_0\epsilon)(I_{2n}+K_0\epsilon)h\\
 &=h^*(I_{2n}-K_0^2\epsilon^2)h
   =h^*h=I_{2n}.
\end{align*}
\end{samepage}
Thus $\widehat T$ is dual complex unitary. Since $S$ is also dual
complex unitary, so is $\widehat W=\widehat T S$.

Next, $hCh^*=C$ by \eqref{eq:structured-pair}. Therefore
\begin{align*}
 \widehat T(C+X\epsilon)\widehat T^*
 &=(I_{2n}+K_0\epsilon)h(C+X\epsilon)h^*
   (I_{2n}-K_0\epsilon)\\
 &=(I_{2n}+K_0\epsilon)
   \bigl(C+hXh^*\epsilon\bigr)(I_{2n}-K_0\epsilon)\\
 &=\bigl[C+hXh^*\epsilon+K_0C\epsilon
   +K_0hXh^*\epsilon^2\bigr](I_{2n}-K_0\epsilon)\\
 &=\bigl[C+(hXh^*+K_0C)\epsilon\bigr]
   (I_{2n}-K_0\epsilon)\\
 &=C-CK_0\epsilon+(hXh^*+K_0C)\epsilon
   -(hXh^*+K_0C)K_0\epsilon^2\\
 &=C+\bigl(hXh^*+K_0C-CK_0\bigr)\epsilon.
\end{align*}
All terms containing $\epsilon^2$ vanish in the last two equalities. By
$\Delta=Y-hXh^*$ and $\Delta=K_0C-CK_0$, the coefficient of $\epsilon$
in the last expression is
\[
 hXh^*+K_0C-CK_0=hXh^*+\Delta=Y.
\]
Consequently,
\[
 \widehat T(C+X\epsilon)\widehat T^*
 =C+Y\epsilon.
\]
Every coefficient of $\widehat W$ is fixed by $\Theta$. It follows that
\[
 B_{\mathbb C}=\widehat W A_{\mathbb C}\widehat W^*.
\]
Since both coefficients of $\widehat W$ are fixed by $\Theta$, the
coefficientwise version of \eqref{eq:fixed-adjoint} yields a dual
quaternion matrix $W$ such that $\widehat W=W_{\mathbb C}$. Furthermore,
\[
 (WW^*)_{\mathbb C}
 =W_{\mathbb C}W_{\mathbb C}^*
 =\widehat W\widehat W^*
 =I_{2n}=(I_n)_{\mathbb C}.
\]
By injectivity of the complex adjoint representation, $WW^*=I_n$; hence
$W$ is unitary. Finally,
\[
 (B-WAW^*)_{\mathbb C}
 =B_{\mathbb C}-\widehat W A_{\mathbb C}\widehat W^*=0.
\]
Injectivity again gives $B=WAW^*$.
This proves the converse.
\end{proof}

For finite families of dual quaternion matrices, we obtain a conclusion similar to that of Lemma \ref{lem:replacement}.

\begin{theorem}
\label{cor:simultaneous-dual}
Let $A_1,\ldots,A_s,B_1,\ldots,B_s\in\mathbb{DQ}^{n\times n}$. The following
statements are equivalent:
\begin{enumerate}
\item [(i)]there is a dual quaternion unitary matrix $W\in\mathbb{DQ}^{n\times n}$ such that
\[
 B_r=WA_rW^*,\qquad r=1,\ldots,s;
\]
\item [(ii)]there is a dual complex unitary matrix $\widehat U\in\mathbb{DC}^{2n\times 2n}$ such that
\[
 (B_r)_{\mathbb C}
 =\widehat U(A_r)_{\mathbb C}\widehat U^*,
 \qquad r=1,\ldots,s.
\]
\end{enumerate}
Moreover, if (ii) holds, then the equalities in (ii) are also realized by a
possibly different dual complex unitary matrix fixed by $\Theta$; equivalently, this matrix has the form $W_{\mathbb C}$
for a dual quaternion unitary matrix $W$.
\end{theorem}

\begin{proof}
The implication (i)$\Rightarrow$(ii) follows immediately
from the multiplicative and adjoint-preserving properties of the dual
complex adjoint representation. We prove the converse.

Write
\[
 A_r=A_{r,st}+A_{r,\mathcal{I}}\epsilon,
 \qquad
 B_r=B_{r,st}+B_{r,\mathcal{I}}\epsilon
\]
and set
\[
 C_r=(A_{r,st})_{\mathbb C},\quad
 X_r=(A_{r,\mathcal{I}})_{\mathbb C},\qquad
 D_r=(B_{r,st})_{\mathbb C},\quad
 Y_r=(B_{r,\mathcal{I}})_{\mathbb C}.
\]
All these matrices are fixed by $\Theta$. By the unitarity of
$\widehat U$, it can be written as
\[
 \widehat U=(I_{2n}+K\epsilon)U_0,
 \qquad U_0^*U_0=I_{2n},\qquad K^*=-K.
\]
Comparing the standard and dual coefficients in \textup{(ii)} gives, for
every $r$,
\begin{align}
 D_r&=U_0C_rU_0^*,\label{eq:sim-standard}\\
 Y_r&=U_0X_rU_0^*+[K,D_r].\label{eq:sim-dual}
\end{align}

Apply Lemma~\ref{lem:replacement} simultaneously to all the pairs
$(C_r,D_r)$. There is a unitary matrix $S$, fixed by $\Theta$, such that
\[
 D_r=SC_rS^*,\qquad r=1,\ldots,s.
\]
Define
\[
 \widetilde X_r=SX_rS^*,
 \qquad g=U_0S^*.
\]
Then $gD_rg^*=D_r$ for every $r$, and \eqref{eq:sim-dual} becomes
\begin{equation}\label{eq:sim-reduced}
 Y_r=g\widetilde X_rg^*+[K,D_r],
 \qquad r=1,\ldots,s.
\end{equation}

To treat all commutator equations at once, regard
\[
 \mathcal M_s=(\mathbb{C}^{2n\times2n})^s
\]
as a real inner-product space with
\[
 \langle(Z_r)_{r=1}^s,(T_r)_{r=1}^s\rangle
 =\sum_{r=1}^s Re~tr(Z_r^*T_r).
\]
Let $\mathcal{K}=\{L\in\mathbb{C}^{2n\times2n}:L^*=-L\}$ and define the joint
commutator map
\[
 \delta_{\mathbf D}:\mathcal{K}\longrightarrow\mathcal M_s,
 \qquad
 \delta_{\mathbf D}(L)=([L,D_1],\ldots,[L,D_s]).
\]
Put
\[
 \mathcal R_{\mathbf D}=\operatorname{ran}\delta_{\mathbf D},
 \qquad
 \mathcal E_{\mathbf D}=\mathcal R_{\mathbf D}^{\perp},
\]
and denote by $\Pi_{\mathbf D}$ the real orthogonal projection onto
$\mathcal E_{\mathbf D}$.

Let $\Theta_s$ act componentwise on $\mathcal M_s$. Because every $D_r$ is
fixed by $\Theta$,
\[
 \Theta_s\bigl(\delta_{\mathbf D}(L)\bigr)
 =\delta_{\mathbf D}(\Theta(L)).
\]
Thus $\Theta_s$ preserves $\mathcal R_{\mathbf D}$ and, being a real
isometry, also preserves $\mathcal E_{\mathbf D}$. Consequently,
\begin{equation}\label{eq:sim-projection-theta}
 \Pi_{\mathbf D}\Theta_s=\Theta_s\Pi_{\mathbf D}.
\end{equation}
Similarly, if a unitary matrix $q$ fixes every $D_r$ under conjugation,
then componentwise conjugation $\operatorname{Ad}_q$ satisfies
\[
 \operatorname{Ad}_q\bigl(\delta_{\mathbf D}(L)\bigr)
 =\delta_{\mathbf D}(qLq^*).
\]
It therefore preserves both orthogonal summands, and
\begin{equation}\label{eq:sim-projection-conjugation}
 \Pi_{\mathbf D}\operatorname{Ad}_q
 =\operatorname{Ad}_q\Pi_{\mathbf D}.
\end{equation}

Use bold letters for tuples, for example
$\widetilde{\mathbf X}=(\widetilde X_1,\ldots,\widetilde X_s)$ and
$\mathbf Y=(Y_1,\ldots,Y_s)$. Equation \eqref{eq:sim-reduced} says
\[
 \mathbf Y
 =\operatorname{Ad}_g(\widetilde{\mathbf X})
  +\delta_{\mathbf D}(K).
\]
Applying $\Pi_{\mathbf D}$ and using
\eqref{eq:sim-projection-conjugation} gives
\begin{equation}\label{eq:sim-projected}
 \Pi_{\mathbf D}(\mathbf Y)
 =\operatorname{Ad}_g
   \bigl(\Pi_{\mathbf D}(\widetilde{\mathbf X})\bigr).
\end{equation}
By \eqref{eq:sim-projection-theta}, every component of both projected
tuples is fixed by $\Theta$.

The unitary $g$ simultaneously conjugates the family consisting of the
$s$ matrices $D_r$ and the $s$ components of
$\Pi_{\mathbf D}(\widetilde{\mathbf X})$ to the corresponding family
consisting of the same $D_r$ and the components of
$\Pi_{\mathbf D}(\mathbf Y)$. Lemma~\ref{lem:replacement}, applied to
these $2s$ pairs, gives a unitary matrix $h$, fixed by $\Theta$, such that
\begin{equation}\label{eq:sim-structured-projected}
 hD_rh^*=D_r\quad(r=1,\ldots,s),\qquad
 \operatorname{Ad}_h
 \bigl(\Pi_{\mathbf D}(\widetilde{\mathbf X})\bigr)
 =\Pi_{\mathbf D}(\mathbf Y).
\end{equation}

Set
\[
 \boldsymbol\Delta
 =\mathbf Y-\operatorname{Ad}_h(\widetilde{\mathbf X}).
\]
Since $h$ fixes every $D_r$, equations
\eqref{eq:sim-projection-conjugation} and
\eqref{eq:sim-structured-projected} imply
\[
 \Pi_{\mathbf D}(\boldsymbol\Delta)=0.
\]
Hence $\boldsymbol\Delta\in\mathcal R_{\mathbf D}$, so there is an
$L\in\mathcal{K}$ with
\[
 \boldsymbol\Delta=\delta_{\mathbf D}(L).
\]
The tuple $\boldsymbol\Delta$ is fixed by $\Theta_s$. Therefore the averaged
matrix
\[
 K_0=\frac{L+\Theta(L)}2
\]
is skew-Hermitian, is fixed by $\Theta$, and satisfies
\begin{equation}\label{eq:sim-structured-commutator}
 \delta_{\mathbf D}(K_0)
 =\frac12\bigl(\boldsymbol\Delta
                 +\Theta_s(\boldsymbol\Delta)\bigr)
 =\boldsymbol\Delta.
\end{equation}

Define
\[
 \widehat T=(I_{2n}+K_0\epsilon)h,
 \qquad
 \widehat W=\widehat T S.
\]
Because $K_0^*=-K_0$, $h$ and $S$ are unitary, and $\epsilon^2=0$, both
$\widehat T$ and $\widehat W$ are dual complex unitary. Their coefficients
are fixed by $\Theta$. For every $r$, the first-order conjugation formula,
\eqref{eq:sim-structured-projected}, and
\eqref{eq:sim-structured-commutator} give
\begin{align*}
 \widehat T(D_r+\widetilde X_r\epsilon)\widehat T^*
 &=D_r+\bigl(h\widetilde X_rh^*+[K_0,D_r]\bigr)\epsilon\\
 &=D_r+Y_r\epsilon.
\end{align*}
Since
\[
 S(C_r+X_r\epsilon)S^*=D_r+\widetilde X_r\epsilon,
\]
we obtain
\[
 (B_r)_{\mathbb C}
 =\widehat W(A_r)_{\mathbb C}\widehat W^*,
 \qquad r=1,\ldots,s.
\]

The coefficientwise fixed-point identity \eqref{eq:fixed-adjoint} gives a
dual quaternion matrix $W$ with $\widehat W=W_{\mathbb C}$. Moreover,
\[
 (WW^*)_{\mathbb C}
 =\widehat W\widehat W^*=I_{2n}=(I_n)_{\mathbb C},
\]
so injectivity of the complex adjoint representation shows that $W$ is
unitary. For each $r$,
\[
 (B_r-WA_rW^*)_{\mathbb C}=0,
\]
and injectivity again gives $B_r=WA_rW^*$. Thus the same dual quaternion
unitary matrix works for the whole family.
\end{proof}

\begin{corollary}\label{cor:scalar-unified}
For any two dual quaternions $A,B\in\mathbb{DQ}$, whether their standard parts are
real or nonreal,
\[
 A\sim_u B\quad\Longleftrightarrow\quad
 A_{\mathbb C}\sim_u B_{\mathbb C}.
\]
Thus no separate construction according to the standard part is required.
\end{corollary}

We now return to Trace words. The following definition applies in arbitrary
order.
\begin{definition}
Suppose that $A=A_{st}+A_{\mathcal{I}}\epsilon\in\mathbb{DQ}^{n\times n}$, where $A_{st},A_{\mathcal{I}}\in\mathbb{Q}^{n\times n}$. Define the Trace of $A$ as
$$
\mathrm{Tr}(A)=\mathrm{Tr}(A_{st})+\mathrm{Tr}(A_{\mathcal{I}})\epsilon.
$$
Similarly, we use the componentwise notation
\[
\operatorname{Tr}(M_1,\ldots,M_k)
:=
\bigl(\operatorname{Tr}(M_1),\ldots,\operatorname{Tr}(M_k)\bigr).
\]
\end{definition}
Similar to the traces of dual complex matrices, the Traces of dual quaternion matrices satisfy the following basic identities.

\begin{proposition}
Suppose that $A=A_{st}+A_{\mathcal{I}}\epsilon,B=B_{st}+B_{\mathcal{I}}\epsilon\in\mathbb{DQ}^{n\times n}$, where $A_{st},A_{\mathcal{I}},B_{st},B_{\mathcal{I}}\in\mathbb{Q}^{n\times n}$. $\lambda\in\mathbb{DR}$. Then
\begin{enumerate}
\item [(1)]$\mathrm{Tr}(A+B)=\mathrm{Tr}(A)+\mathrm{Tr}(B)$.
\item [(2)]$\mathrm{Tr}(AB)=\mathrm{Tr}(BA)$.
\item [(3)]$\lambda \mathrm{Tr}(A)=\mathrm{Tr}(\lambda A)$.
\item [(4)]If $B=P^{-1}AP$ for some invertible matrix $P\in\mathbb{DQ}^{n\times n}$, then $\mathrm{Tr}(A)=\mathrm{Tr}(B)$.
\item [(5)]$\mathrm{Tr}(A^{T})=\mathrm{Tr}(A)$.
\item [(6)]$\mathrm{Tr}(\overline{A})=\overline{\mathrm{Tr}(A)}$.
\item [(7)]$\mathrm{Tr}(A^{\ast})=\overline{\mathrm{Tr}(A)}$.
\end{enumerate}
\end{proposition}
By the definition of complex trace and quaternion Trace, it can be seen that for any $A\in\mathbb{DQ}^{n\times n}$
$$
\mathrm{Tr}(A)=\mathrm{tr}(A_{\mathbb{C}}).
$$

The arbitrary-order adjoint criterion does not by itself supply a finite
Trace-word criterion in arbitrary order. For scalar dual quaternions,
however, the $2\times2$ dual complex result from Section~4 gives the
following conclusion.

\begin{proposition}
Suppose that $A=A_{st}+A_{\mathcal{I}}\epsilon,B=B_{st}+B_{\mathcal{I}}\epsilon\in\mathbb{DQ}$ and $A_{st}$ is not real. $A_{\mathbb{C}}$ and $B_{\mathbb{C}}$ are complex adjoint matrices of $A$ and $B$, respectively. Then the following statements are equivalent.
\begin{enumerate}
\item $\mathrm{tr}(A_{\mathbb{C}},A_{\mathbb{C}}^{2},A_{\mathbb{C}}^{\ast}A_{\mathbb{C}})=\mathrm{tr}(B_{\mathbb{C}},B_{\mathbb{C}}^{2},B_{\mathbb{C}}^{\ast}B_{\mathbb{C}})$.
\item $A_{\mathbb{C}}$ and $B_{\mathbb{C}}$ are unitarily similar.
\end{enumerate}
\end{proposition}

\begin{proof}
The implication $(2)\Rightarrow(1)$ follows immediately from the
invariance of the trace under unitary similarity.

Conversely, assume (1), and write
\[
A_{\mathbb{C}}=C+X\epsilon,\qquad B_{\mathbb{C}}=D+Y\epsilon,
\]
where $C=(A_{\rm st})_{\mathbb{C}}$ and $D=(B_{\rm st})_{\mathbb{C}}$. Both $C$ and $D$
are normal. Since $A_{\rm st}$ is nonreal, $C$ has two distinct
eigenvalues. Comparing the standard parts of
$\operatorname{tr}(A_{\mathbb{C}})=\operatorname{tr}(B_{\mathbb{C}})$ and
$\operatorname{tr}(A_{\mathbb{C}}^2)=\operatorname{tr}(B_{\mathbb{C}}^2)$ shows that
$C$ and $D$ have the same characteristic polynomial. Hence they
are unitarily similar.

By Lemma \ref{hu}, the unitary similarity of the standard parts may be
realized by a unitary matrix of complex adjoint form. After this
change of basis, the infinitesimal parts still have complex-adjoint
form. Consequently, in the notation of Remark \ref{55}, their entries
satisfy
\[
a_4=\overline{a_1},\qquad
b_3=-\overline{b_2},\qquad
d_3=-\overline{d_2}.
\]
The dual parts of the three trace identities give the remaining
conditions in Remark \ref{55}, while the above relations make
condition (\ref{4.25}) hold. Therefore, by Theorem \ref{14}, $A_{\mathbb{C}}$ and $B_{\mathbb{C}}$
are unitarily similar.
\end{proof}

\begin{corollary}\label{999}
Suppose that $A=A_{st}+A_{\mathcal{I}}\epsilon,B=B_{st}+B_{\mathcal{I}}\epsilon\in\mathbb{DQ}$ and $A_{st}$ is not real. Then $A$ is unitarily similar to $B$ if and only if 
$$
\mathrm{Tr}(A,A^{2},A^{\ast}A)=\mathrm{Tr}(B,B^{2},B^{\ast}B).
$$
\end{corollary}

\begin{example}
Let $A=1+(1+x\boldsymbol{j})\epsilon$ where $x$ is a nonzero complex number, and let $B=1+\epsilon$. For any binary word $W(a, b)$, where $a$ appears $m$ times and $b$ appears $n$ times, one can see that
\[
W(A,A^{\ast})=1+((m+n)+(m-n)x\boldsymbol{j})\epsilon\ \ \ \text{and} \ \  \
W(B,B^{\ast})=1+(m+n)\epsilon.
\]
Consequently,
$$
\mathrm{Tr}(W(A,A^{\ast}))=2+2(m+n)\epsilon=\mathrm{Tr}(W(B,B^{\ast})).
$$
Therefore, all the Trace words of $A$ and $B$ are equal.
However, $A$ is not unitarily similar to $B$. Since $A_{st}=B_{st}$ is real and $A_{\mathcal{I}}$ is not unitarily similar to $B_{\mathcal{I}}$, it follows from Proposition \ref{99} that $A$ and $B$ are not unitarily similar.
\end{example}

\begin{remark}
Theorem~\ref{thm:arbitrary-order} is valid for every matrix order, whereas
the finite Trace-word statement in Corollary~\ref{999} is a
scalar consequence of the $2\times2$ dual complex classification. A finite
Trace-word classification for arbitrary $n$ would require an additional
completeness theorem for $2n\times2n$ dual complex matrices; it does not
follow solely from the complex adjoint representation.
\end{remark}

\end{document}